\documentclass[11pt]{article}
\usepackage{amssymb}
\usepackage{amsmath}
\usepackage{mathrsfs}
\def\.{\partial_t}
\def\:{\partial_{tt}}
\def\d{\partial}

\def\D{\Delta}
\def\g{\nabla}
\def\hf{\frac{1}{2}}

\def\a{\alpha}
\def\e{\varepsilon}
\def\eps{\varepsilon}
\def\f{\varphi}

\def\l{\lambda}
\def\la{\lambda}

\def\s{\sigma}
\def\t{\tau}

\def\O{\Omega}

\def\ga{\gamma}

\def\Ac{\mathfrak{A}}
\def\cA{\mathcal{A}}
\def\cB{\mathcal{B}}
\def\sB{\mathscr{B}}

\def\sD{\mathscr{D}}

\def\sK{\mathscr{K}}
\def\cL{\mathcal{L}}
\def\cE{\mathcal{E}}
\def\H20{\mathcal{H}_{2,0}}
\def\R{\mathbb{R}}
\def\cR{\mathcal{R}}

\def\cN{\mathcal{N}}
\def\cH{\mathcal{H}}
\def\cW{\mathcal{W}}
\def\cO{\mathcal{O}}

\def\N{\mathbb{N}}
\def\<{\langle}
\def\>{\rangle}
\def\8{\infty}

\oddsidemargin
0.1cm \textwidth 16.2cm

\newtheorem{lemma}{Lemma}[section]
\newtheorem{theorem}[lemma]{Theorem}
\newtheorem{remark}[lemma]{Remark}
\newtheorem{proposition}[lemma]{Proposition}
\newtheorem{corollary}[lemma]{Corollary}
\newtheorem{definition}[lemma]{Definition}
\newtheorem{assumption}[lemma]{Assumption}

\newcommand{\wrt}{{with respect to }}

\newenvironment{declaration}[1]{\trivlist
\item[\hskip \labelsep{\bf #1 }]\ignorespaces}{\endtrivlist}
\newenvironment{proofof}[1]{\begin{declaration}{#1}}{\hfill
$\square$\end{declaration}}
\newenvironment{proof}{\begin{proofof}{Proof.}}{\end{proofof}}

\begin{document}

\title{Long-time dynamics of Kirchhoff wave models \\ with strong nonlinear  damping}
\author{Igor
Chueshov\thanks{e-mail:
chueshov@univer.kharkov.ua} 
\\ \\
Department of Mechanics and Mathematics, \\ Kharkov
National University, \\ Kharkov, 61077,  Ukraine
 }
 \maketitle

\begin{abstract}
We study well-posedness and long-time dynamics of a class of
 quasilinear wave equations with 
a strong damping. We accept the Kirchhoff  hypotheses and assume that
the stiffness and damping coefficients are $C^1$ functions of the 
$L_2$-norm of the gradient of the displacement.  
We first prove the existence and uniqueness of weak solutions
and study their properties for a rather wide class of nonlinearities
which covers the case of possible degeneration (or even negativity)
of the  stiffness coefficient 
 and the case  of a supercritical source term.
Our main results deal with global attractors.
In the case of  strictly positive stiffness factors we prove that in 
the natural energy  space endowed with a partially strong topology
there exists  a global attractor whose 
fractal dimension is finite.
In the non-supercritical case the partially strong topology becomes
strong and  a finite dimensional 
attractor exists in the strong topology of the energy space.
Moreover, in this case we also establish  the existence of a fractal exponential 
 attractor and give conditions that guarantee the existence 
of a finite number of determining functionals.
Our arguments   involve a recently
developed   method based on ``compensated'' compactness and
quasi-stability estimates. 
\end{abstract}
\section{Introduction}

In a bounded smooth domain $\O\subset\R^d$ we consider the following 
Kirchhoff wave model with a strong nonlinear damping: 
\begin{equation}\label{main_eq}
    \left\{ \begin{array}{l}
        \:u -\s(\|\g u\|^2)\D\.u - \phi(\|\g u\|^2)\D u+ f(u)=h(x),~~ x\in\O,\; t>0,\\ [2mm]
u|_{\d\O}=0,~~ u(0)=u_0, \quad \.u(0)=u_1.  \\
    \end{array} \right.
\end{equation}
Here $\D$ is the Laplace operator, $\s$ and $\phi$ are scalar
functions  specified later, $f(u)$ is a  given source term,
$h$ is a given function in $L^2(\O)$ and $\|\cdot\|$ is the norm in $L^2(\O)$.
\par 
This kind of wave
 models goes back to G. Kirchhoff ($d=1$, $\phi(s)=\f_0+\f_1 s$,
$\s(s)\equiv 0$, $f(u)\equiv 0$) and  has been studied by many authors
under different types of hypotheses. We refer to
\cite{bernstein40,Lions77,pochozhaev} and  to
 the literature
cited in the  survey ~\cite{MFM-survey2002},
see also \cite{CCFS98, Fan-Zh04,ghisi06,Hashimoto-Yam-jde07, MatIke96,Nakao-09,Nakao-Z-07,
Ono97a,Ono97b,Yang-Z07,ZY-W-jde2010,ZY-L-jmaa2010,
Yamazaki-jde05} and the references therein.
\par 
Our main goal in this paper is to study well-posedness and 
long-time dynamics of the problem  (\ref{main_eq}) under 
 the following set of  hypotheses:
\begin{assumption}\label{A1}
{\rm
\begin{enumerate}
  \item[{\bf (i)}]  The damping  ($\s$) and  the stiffness ($\phi$)
factors  are $C^1$  functions on the semi-axis $\R_+= [0,+\infty)$.
Moreover,  $\s(s)>0$ for all $s\in\R_+$  and there exist
$c_i\ge 0$ and $\eta_0\ge 0$  such that  
\begin{equation}\label{phi-si-coer}
\int_0^s \left[\phi(\xi)+\eta_0 \s(\xi)\right]d\xi \to+\infty~~
\mbox{as}~~
s\to+\infty 
\end{equation}
and
\begin{equation}\label{phi-s-si}
s\phi(s)+c_1\int_0^s \s(\xi)d\xi \ge -c_2~~\mbox{for}~~s\in\R_+.
\end{equation}
   \item[{\bf (ii)}] $f(u)$ is a $C^1$ function 
such that $f(0)=0$ (without loss of generality), 
\begin{equation}\label{f-coercive}
\mu_f:=\liminf_{|s|\to\infty}\left\{ s^{-1}f(s)\right\}>-\infty,
\end{equation}
and the following  properties hold:
(a) 
if $d=1$, then $f$ is arbitrary; (b) if $d=2$ then
\begin{equation*}
|f'(u)|\le C\left(1+|u|^{p-1}\right) ~~\mbox{for some}~~p\ge 1; 
\end{equation*}
(c) if $d\ge 3$ then either 
\begin{equation}\label{f-crit}
|f'(u)|\le C\left(1+|u|^{p-1}\right)\quad\mbox{with some}~~ 
1\le p\le p_*\equiv \frac{d+2}{d-2},
\end{equation}
or else
\begin{equation}\label{f-supercrit}
 c_0|u|^{p-1}  -c_1 \le f'(u)\le c_2\left(1+|u|^{p-1}\right) 
\quad\mbox{with some}~~p_*< p< p_{**}\equiv  \frac{d+4}{(d-4)_+},
\end{equation}
where $c_i$ are positive constants and $s_+=(s+|s|)/2$.
\end{enumerate}
}
\end{assumption}
\begin{remark}\label{re:relax}
{\rm {\bf (1)} 
The coercive behavior in (\ref{phi-si-coer})  and (\ref{phi-s-si}) holds
with $\eta_0=c_1=0$
if  we assume that $ \liminf_{s\to+\infty}\left\{ s \phi(s)\right\} >0$,
for instance. The standard  example 
is $\phi(s)=\phi_0+\phi_1s^\alpha$ 
with  $\phi_0\in\R$, $\phi_1>0$
and $\alpha\ge 1$. However we can also take    $\phi(s)$ 
with finite support, or even  $\phi(s)\equiv const\le0$.
In this case we need additional hypotheses 
concerning behavior  of $\s(s)$ as $s\to+\infty$.
We note that the physically justified situation 
(see, e.g., the survey \cite{MFM-survey2002}) corresponds to
the case when the stiffness coefficient $\phi(s)$ is 
positive almost everywhere. However we include 
into the consideration the case of possibly negative $\phi$
because  the argument we use 
to prove well-posedness involves  positivity properties
of $\phi$ in a rather mild form (see, e.g., (\ref{phi-si-coer}) and
(\ref{phi-s-si})). 
\par {\bf (2)} 
We note that in the case when $d\le 2$ or $d\ge 3$ and (\ref{f-crit})
holds with $p<p_*$ the Nemytski operator $u\mapsto f(u)$
is a locally Lipschitz mapping from the Sobolev space $H^1_0(\O)$
into $H^{-1+\delta}(\O)$ for some $\delta>0$.
If  $d\ge 3$ and (\ref{f-crit})
holds with $p=p_*$ this fact is valid with $\delta=0$.
These properties of the source nonlinearity $f(u)$ are of importance
in the study of wave dynamics with the strong damping
(see, e.g., \cite{CaCh02,ChDl06,PataZelik06,YaSu09} and the references 
therein). Below we refer to this situation as to 
non-supercritical (subcritical when $\delta>0$ and critical
for the case $\delta=0$).
To deal with the supercritical case (the inequality in 
(\ref{f-crit})  holds with $p>p_*$)
we borrow some ideas from \cite{KaZe09} and we need 
a lower bound for $f(u)$ of the same order as its upper bound 
(see the requirement in (\ref{f-supercrit})).
The second critical exponent $p_{**}$ arises  in the dimension 
$d\ge 5$ from 
the requirement $H^2(\O)\subset L_{p+1}(\O)$ which we need 
to estimate the source term in some negative Sobolev space,
see also Remark~\ref{re-mult} below. 
\par {\bf (3)} 
We also note that in the case (\ref{f-supercrit}) 
the condition in (\ref{f-coercive}) holds automatically
(with $\mu_f=+\infty$). 
This condition  can be relax depending 
on the properties of $\phi$. For instance, 
in the case when $\phi(s)=\phi_0+\phi_1s^\alpha$ 
with  $\phi_1>0$ instead of (\ref{f-coercive}) we can assume that
\[
f(s)s\ge -c_1 |s|^l-c_2\quad \mbox{for some $l\le 
 \min\{2\a+2-\eps, 2d/(d-2)_+\}$}
\]
with arbitrary small $\eps>0$.
Therefore for this choice of $\phi$ we need no coercivity assumptions
concerning $f$ in the non-supercritical case provided $p<2\a+1$.
However we do not pursue these possible generalizations and 
prefer to keep hypotheses concerning $\phi$ and $\s$ as general as possible.  
}
\end{remark}
\par
Well-posedness issues for Kirchhoff type models like 
(\ref{main_eq}) were studied intensively last years.
The main attention was paid the case when the strong damping 
term $-\s\D u_t$ is absent and the source term $f(u)$
 is either absent or subcritical. We refer to  \cite{ghisi06,Ono97a,Yamazaki-jde05}
and also to the survey \cite{MFM-survey2002}.
In these papers the authors have studied 
sets of initial data for which solutions exist and are unique.
The papers \cite{ghisi06,Ono97a} consider also the case
of a degenerate stiffness coefficient ($\phi(s)\sim s^\alpha$ near zero).
We also mention  the paper \cite{MatIke96} 
which deals with global existence (for a restricted class of initial data)
  in the case of   a strictly positive stiffness factor
of the form $\phi(s)=\phi_0+\phi_1 s^\a$ with
the nonlinear damping  $|u_t|^q u_t$ and 
the source term $f(u)=-|u|^p u$  for some range of exponents 
$q$ and $p$, see also the recent paper
\cite{taniguchi-jmaa2010} which is concentrated on the {\em local}
existence issue    for the same type of   damping 
and source terms but for 
 a wider range of the exponents $p$ and $q$.
 \par 
Introducing of the strong (Kelvin-Voigt) damping term $-\s\D u_t$
provides an additional a priori estimate and simplifies the 
issue. There are several well-posedness results  available 
in the literature for this case (see 
\cite{CCFS98, Med-Milla-M-1990,Nakao-Z-07,Ono97b,Yang-Z07,ZY-L-jmaa2010,ZY-W-jde2010}).
However all these publications assume that the damping coefficient 
$\s(s)\equiv \s_0>0$ is a constant and deal with a subcritical 
or absent source term. Moreover, all of them 
(except \cite{Ono97b}) assume that  stiffness  factor 
is non-degenerate  (i.e., $\phi(s)\ge \phi_0>0$). However
\cite{Ono97b} assumes small initial energy, i.e., deals with local (in phase
space) dynamics.
Recently  the existence and uniqueness of weak (energy) solutions 
of (\ref{main_eq}) was reported (without detailed proofs)
in \cite{KaZe09} for the case of supercritical source
satisfying (\ref{f-supercrit}). However the authors in \cite{KaZe09}
assume (in addition to our hypotheses) that  $d=3$, the damping
is linear (i.e.,\ $\s(s)=const>0$) and the stiffness factor $\phi$
is a uniformly positive $C^1$ function  
satisfying the inequality $\int_0^s\phi(\xi)d\xi\le s\phi(s)$
for all $s\ge 0$. As for {\em nonlinear  strong} damping to the 
best of our knowledge  there is only one publication   \cite{lazo08}. 
This paper deals with nonlinear  damping of the  form
$\s(\|A^\alpha u\|^2) A^\alpha u_t$ with $0<\alpha\le 1$. 
The main result of   \cite{lazo08} 
states only the existence of weak solutions 
for uniformly positive $\phi$ and $\s$ in the case when $f(u)\equiv 0$.
\par 
The main achievement of our well-posedness result is that
(a) we do not assume any kind of non-degeneracy conditions
concerning $\phi$ (this function may be  zero or even negative);
(b) we consider a nonlinear state-dependent strong damping and
 do not assume uniform positivity of the
damping factor $\s$; 
(c) we cover the cases of critical and supercritical
source terms $f$.

\smallskip 
\par 
Our second result deals with a global attractor 
for the dynamical system generated by (\ref{main_eq}).
There are many papers on stabilization 
to zero equilibrium  for Kirchhoff type models (see, e.g., 
\cite{APS-NA-RWA2009,CCFS98,MatIke96,MFM-survey2002}
and the references therein) and only
a few recent results devoted to (non-trivial) attractors
for systems like (\ref{main_eq}). We refer to \cite{Nakao-09} 
for  studies of local attractors in the case
of viscous damping
and to \cite{Fan-Zh04,Nakao-Z-07,Yang-Z07,ZY-W-jde2010,ZY-L-jmaa2010}
in the case of a strong {\em linear} damping (possibly perturbed by nonlinear
viscous terms).
 All these papers assume  
  subcriticality  of the  force $f(u)$ and deal with 
a uniformly positive 
stiffness coefficient of the form $\phi(s)=\phi_0+\phi_1 s^\a$ with $\phi_0>0$. 
In the long time dynamics context we can point only
the paper \cite{APS-NA-RWA2009}  which contains a result 
(see Theorem 4.4\cite{APS-NA-RWA2009}) on stabilization to zero
 in the case when $\phi(s)\equiv\s(s)= a+b s^\gamma$  with $a>0$
and possibly supercritical  source with the property
$f(u)u+a\mu u^2\ge 0$, where $\mu>0$  is small enough.
In this case the global attractor $\Ac=\{ 0\}$ is trivial. 
However this paper does not discuss well-posedness issues 
and assumes the existence of sufficiently  smooth solutions 
 as a starting point of
the whole considerations.
\par 
 Our main novelty is that
 we consider long-time dynamics for much more  general 
stiffness and damping coefficients
and cover the supercritical case.
Namely, under some additional non-degeneracy assumptions 
we prove the existence of a finite dimensional
 global attractor which uniformly attracts
trajectories in  a partially strong sense (see Definition~\ref{de:ss-top}). 
In the non-supercritical case this result can be improved:
we establish the convergence property with respect to strong topology of
the phase (energy) space. Moreover, in this case we prove the existence
of a fractal exponential attractor and give conditions for the existence
of finite sets of determining functionals.
To establish these results we rely on recently developed approach (see \cite{ChuLas_JDDE_2004} and also  \cite{ChuLas} and   
\cite[Chapters 7,8]{cl-book}) which  involves stabilizability estimates,
the notion of a quasi-stable system and also the idea 
of "short" trajectories
due to \cite{malek-ne,malek}.  In the supercritical case
to prove that the attractor has a finite dimension 
we also use a recent observation made in \cite{KaZe09} concerning
stabilizability estimate in the extended space.
In the non-supercritical case
we first prove that the corresponding system is quasi-stable 
in the sense of the  definition given in \cite[Section 7.9]{cl-book}
and then  apply the general theorems
on properties of quasi-stable  systems from this source.
\par 
We also note that
long-time dynamics of second order equations with nonlinear damping 
was studied by many authors. We refer to 
\cite{LasieckaBarbuRammaha,ChuKol, GattiPata_1D, Kolbasin,
 PataZelik_2D, PataZelik_3D} for the 
case of a damping   with
a displacement-dependent coefficient   and
to  \cite{ChuLas_JDDE_2004,ChuLas,cl-book} 
and to the references therein
for a velocity-dependent damping.
Models with different types of strong (linear) damping
in wave equations were considered in 
\cite{CaCh02,ChDl06,KaZe09,PataZelik06,YaSu09},
see also the literature quoted in these references. 
\par  
The paper is organized as follows.
 In Section~\ref{sect2} we introduce some notations and
prove Theorem ~\ref{pr:wp1} which 
provides us with well-posedness of our  model and 
contains some additional properties of solutions.
In Section~\ref{Sect3} we study long-time dynamics of the 
evolution semigroup $S(t)$ generated by (\ref{main_eq}).
We first establish some continuity properties of $S(t)$
(see Proposition~\ref{pr:gener}) and its   dissipativity
(Proposition~\ref{pr:dis}). These results do not require 
any non-degeneracy hypotheses concerning the stiffness coefficient $\phi$.
Then in the case of strictly positive $\phi$ we prove 
asymptotic compactness of $S(t)$ (see Theorem~\ref{th:ak}
and Corollary~\ref{co:ak}).
Our main results in Section~\ref{Sect3} state the existence of
 global
attractors and describe their properties 
in both the general case (Theorems~\ref{th1:attractor} and \ref{th:dim})
and the non-supercritical case (Theorems~\ref{th:attractor}
and \ref{th:exp-det}).

\section{Well-posedness}\label{sect2}
 We first describe some notations.
\par
Let  $H^\s(\O)$ be the  $L_2$-based Sobolev space
of the order $\s$ with the norm denoted by $\|\cdot\|_\s$
and
$H^\s_0(\O)$ is the completion of $C^\infty_0(\O)$ in $H^\s(\O)$
for $\s>0$.
Below we also
 denote by  $\|\cdot\|$ and $(\cdot,\cdot)$  the norm and the inner product
in $L_2(\O)$.
\par
In the space $H=L_2(\O)$ we introduce the operator
$\cA=-\Delta_D$ with the domain  
$$ \sD ( \cA)=\left\{ u\in H^2(\O) \, : \;
 u=0 ~~{\rm on}~~\O\right\} \equiv H^2(\O)\cap H^1_0(\O),
$$
where
$\Delta_D$ is the Laplace operator in $\O$ with the Dirichlet boundary conditions.
 The operator  $\cA$ is a  linear 
self-adjoint positive operator densely
defined on $H=L_2(\O)$. 
The resolvent of $\cA$ is compact in $H$.
Below we  denote by $\{e_k\}$ the orthonormal basis in $H$ consisting of
eigenfunctions  of the operator $\cA$:
\[
\cA e_k=\l_k e_k, 
\quad 0<\l_1\le \l_2\le\cdots, \quad \lim_{k\to\infty}\l_k =\8.
\]
We also
denote
$\cH=[H_0^1(\O)\cap L_{p+1}(\O)]\times L_2(\O)$. In the 
non-supercritical case
 (when $d\le 2$ or $d\ge 3$ and  $p\le p_*=(d-2)(d+2)^{-1}$)
we have that $H_0^1(\O)\subset L_{p+1}(\O)$\footnote{
To unify the presentation we suppose
that $p\ge 1$ is arbitrary in all appearances  in the case $d=1$.
} 
 and thus the space $\cH$
coincides 
with $H_0^1(\O)\times L_2(\O)$.
We define the norm in $\cH$  by the relation
\begin{equation}\label{cH-norm}
\|(u_0; u_1)\|_\cH^2=\|\g u_0\|^2+\a\|u_0\|_{L_{p+1}(\O)}^2+
\|u_1\|^2,
\end{equation}
where $\alpha=1$ in the case when  $d\ge 3$ and $p>p_*$ 
and $\alpha=0$ in other cases. 
\par
\begin{definition}\label{de:weak}
{\rm
A function $u(t)$ is  said to be a weak solution 
to (\ref{main_eq}) on an interval $[0,T]$ if
\begin{equation}\label{u-ut-d}
u \in L_\8(0,T; H^1_0(\O)\cap L_{p+1}(\O) ),~~~
\.u \in L_\8(0,T; L_2(\O))\cap L_2(0,T;H^1_0(\O)) 
\end{equation}
  and (\ref{main_eq}) is satisfied in the sense of distributions.
}
\end{definition} 
\par
Our main result in this section is Theorem~\ref{pr:wp1}
 on well-posedness of problem (\ref{main_eq}). 
This theorem also contains some 
auxiliary  properties of solutions which we 
need for the results on the asymptotic dynamics.
\begin{theorem}[Well-posedness]\label{pr:wp1}
Let Assumption~\ref{A1} be in force and
 $(u_0; u_1) \in \cH$.
Then  for every $T>0$ problem (\ref{main_eq}) has a unique weak solution
$u(t)$ on $[0,T]$. This solution possesses the following properties:
\begin{enumerate}
  \item[{\bf 1.}] The function $t\mapsto (u(t);u_t(t))$ is (strongly) continuous 
in $\cH=[H_0^1\cap L_{p+1}](\O)\times L_2(\O)$
and 
\begin{equation}\label{u-tt}
u_{tt}\in L_2(0,T; H^{-1}(\O))+ L_{\infty}(0,T; L_{1+1/p}(\O)).
\end{equation}
Moreover, there exists a constant $C_{R,T}>0$ such that
\begin{equation}\label{basic-bnd}
\| u_{t}(t)\|^2+ \| \g u (t)\|^2 
+ c_0\|u(t)\|_{L_{p+1}(\O)}^2 + \int_0^{t}\|\g u_{t}(\t)\|^2 d\t
\le C_{R,T}
\end{equation}
for every $t\in [0,T]$ and  initial data $\|(u_0; u_1)\|_\cH\le R$,
where $c_0=1$ in the case when (\ref{f-supercrit}) holds and 
$c_0=0$ in other cases.
We also have the following additional regularity: 
\[
u_t\in L_\infty(a,T; H_0^1(\O)),\quad 
u_{tt}\in L_\infty (a,T; H^{-1}(\O))\cap L_{2}(a,T; L_{2}(\O))
\]
for every $0<a< T$ and there exist $\beta>0$ and $c_{R,T}>0$ such that
\begin{equation}\label{smoth-prop}
\| u_{tt}(t)\|_{-1}^2+ \| \g u_{t}(t)\|^2 
+\! \int_t^{t+1}\!\!\left[\|u_{tt}(\t)\|^2+
c_0\int_\O|u(x,\t)|^{p-1}|u_t(x,\t)|^2dx\right]  d\t
\le \frac{c_{R,T}}{t^\beta}
\end{equation}
for every $t\in (0,T]$, where as above
 $\|(u_0; u_1)\|_\cH\le R$ and
$c_0>0$ in the supercritical case only.
  \item[{\bf 2.}]
 The following energy identity
\begin{equation}\label{8.1.4}
\cE(u(t), u_t(t))+\int_s^t \s(\|\g u(\tau)\|^2) \|\g u_t(\tau)\|^2 d\tau=
\cE(u(s), u_t(s))
\end{equation}
holds for every $t>s\ge 0$, where the energy $\cE$ is defined by the relation
\begin{equation*}
\cE(u_0, u_1)=\frac12\left[ \|u_1\|^2 + 
\Phi\left(\|\g u_0\|^2\right)\right]+\int_\O  F(u_0)dx-\int_\O h u_0 dx,
~~ (u_0;u_1)\in \cH,
\end{equation*}
with 
\[
\Phi(s)=\int_0^s\phi(\xi) d\xi~~and ~~ F(s)=\int_0^s f(\xi) d\xi.
\]
  \item[{\bf 3.}]
If $u^1(t)$ and $u^2(t)$ are two weak solutions
such that $\|(u^i(0); u^i_t(0))\|_\cH\le R$, $i=1,2$, then
there exists $b_{R,T}>0$ such that the difference $z(t)=u^1(t)-u^2(t)$ satisfies the relation
\begin{equation}
\label{dif-bnd}
\|z_t(t)\|^2_{-1}+ \|\g z(t)\|^2 +\int_0^t\|z_t(\t)\|^2 d\t \leq b_{R,T}\left(
 \|z_t(0)\|^2_{-1}+ \|\g z(0)\|^2\right) 
\end{equation}
for all $t\in [0,T]$,
and, if (\ref{f-supercrit}) holds, we also have that 
\begin{equation}
\label{dif-bnd+}
\int_0^T\left[
\int_\O|z|^{p+1}dx +\int_\O (|u^1|^{p-1}+|u^2|^{p-1}) |z|^2 dx \right] d\t \leq b_{R,T}\left(
 \|z_t(0)\|^2_{-1}+ \|\g z(0)\|^2\right).
\end{equation}
  \item[{\bf 4.}] If we assume in addition that $u_0\in (H^2\cap H^1_0)(\O)$,
then 
$u\in C_w(0,T;(H^2\cap H^1_0)(\O))$, where $C_w(0,T; X)$ stands for 
the space of weakly continuous functions with values in $X$,
and  under the condition  $\|(u_0; u_1)\|_\cH\le R$
we have that 
\begin{equation}
\label{sm-bnd}
\|u_t(t)\|^2+ \|\D u(t)\|^2 \leq C_R(T)\left( 1+
 \|\D u_0\|^2\right) \quad\mbox{for every}~~t\in [0,T]. 
\end{equation}
\end{enumerate}
\end{theorem}
\begin{proof} 
Let  $\Sigma(s)=\int_0^s\s(\xi) d\xi$.
For every $\eta>0$
we  introduce the following functional on $\cH$:
\begin{equation}\label{e+eta}
\cE_+^\eta(u_0, u_1)= \|u_1\|^2 + 
\left[\Phi\left(\|\g u_0\|^2\right)+
\eta \Sigma\left(\|\g u_0\|^2\right)
-a(\eta)\right]
+  \a \|u_0\|^{p+1}_{L_{p+1}(\O)}
+\|u_0\|^2
\end{equation}
with
$a(\eta)=\inf_{s\in\R_+}\{\Phi(s)+\eta\Sigma(s)\}$, where
 $\alpha=1$ in the case when (\ref{f-supercrit}) holds and
$\alpha=0$ in other cases.
By (\ref{phi-si-coer}) this functional is finite 
for every $\eta\ge \eta_0$. 
\par
Let $\nu\in\R_+$ and 
\begin{equation}\label{w+eta}
\cW^{\eta,\nu}(u_0, u_1)= \cE(u_0,u_1)
+\eta \left[ (u_0,u_1)+\hf \Sigma\left(\|\g u_0\|^2\right)\right]
+\nu \|u_0\|^2.
\end{equation}
One can see that
for every $\eta\ge \eta_0$ we can choose $\nu=\nu(\eta,\mu_f)\ge 0$,
 positive constants  $a_i$ and 
a monotone positive function $M(s)$ such that  
\begin{equation}\label{e+w-eta}
a_0 \cE_+^\eta(u_0, u_1)-a_1 \le \cW^{\eta,\nu}(u_0, u_1)\le
 a_2 \cE_+^\eta(u_0, u_1) +M(\|\g u_0\|^2),
~~\forall\,  (u_0; u_1)\in\cH.
\end{equation}
\par 
To prove the existence of solutions, we use the standard Galerkin method.
We start with the case when  $u_0\in (H^2\cap H^1_0)(\O)$ and
 assume that $\|(u_0;u_1)\|_\cH\le R$ for some $R>0$.
We seek for approximate solutions  of the form
$$
u^N(t)=\sum\limits_{k=1}^N g_k(t)e_k,\quad N=1,2,\ldots,
$$
that satisfy the finite-dimensional projections of (\ref{main_eq}). 
Moreover, we assume that 
\[
\|(u^N(0);u^N_t(0))\|_\cH\le C_R~~~\mbox{and}~~~ 
\|u^N(0)-u_0\|_2\to0~~\mbox{as}~~ N\to\infty.
\]
Such solutions exist (at least locally), and after multiplication 
of the corresponding projection of
 (\ref{main_eq}) by $u_t^N(t)$ we get that $u^N(t)$ satisfies the energy relation in \eqref{8.1.4}.
Similarly, one can see  from (\ref{phi-s-si}) and (\ref{f-coercive}) that 
\begin{align*}
 \frac{d}{dt}\left[ (u^N,u^N_t)+\frac12
\Sigma(\|\g u^N\|^2)\right]& = \| u^N_t\|^2-  \phi(\|\g u^N\|^2)\|\g u^N\|^2-
(f(u^N), u^N)+(h,u^N)\\
&\le  \| u^N_t\|^2+ C_1\Sigma(\|\g u^N\|^2)+C_2\|u^N\|^2 +C_3.
\end{align*}
One can see from (\ref{phi-si-coer})  that for every $\eta>\eta_0$
there exist $c_i>0$ such that
\[
\Sigma(s)\le c_1\left[\Phi(s)+\eta \Sigma(s)-a(\eta)\right] +c_2,
 ~~~ s\in\R_+.
\]
Thus  using (\ref{e+w-eta}) we have  that 
 the function $\cW^{\eta,\nu}_N(t)\equiv\cW^{\eta,\nu} (u^N(t),u^N_t(t))$
satisfies the inequality
\[
\frac{d}{dt}\cW^{\eta,\nu}_N(t) \le \eta\left(
 \| u^N_t\|^2+ C_1\Sigma(\|\g u^N\|^2)+C_2
\|u^N\|^2 +C_3
\right) \le c_1 \cW^{\eta,\nu}_N(t) +c_2
\]
for $\eta>\eta_0$ with $\nu$ depending on $\eta$ and $f$. 
Therefore, using Gronwall's type argument and also relation (\ref{e+w-eta})
we obtain
\[
\cE^\eta_+(u^N(t);u^N_t(t)) \le C_{R,T}~~\mbox{ for all}~~ t\in [0,T],~~ N=1,2,3\ldots,
\]
for every $\eta>\eta_0$. By the coercivity requirement 
in (\ref{phi-si-coer}) we conclude that
\begin{equation}\label{1st-apri}
\|(u^N(t);u^N_t(t))\|_\cH \le C_{R,T}~~\mbox{ for all}~~ t\in [0,T],~~ N=1,2,3\ldots.
\end{equation}
Since $\s(s)>0$, this implies that $\s(\|\g u^N(t)\|^2)>\s_{R,T}$ for all
$t\in [0,T]$. Therefore 
the energy relation \eqref{8.1.4} for $u^N$ yields that
\begin{equation}\label{2st-apri}
\int_0^T\|\g u_t^N(t)\|^2dt \le C(R,T), \quad N=1,2,\ldots,
~~\mbox{
for any $T>0$.}
\end{equation}
Now we use the multiplier $-\D u$ 
(below we omit the superscript $N$ for shortness).
We obviously have that
\begin{align}\label{20a}
\lefteqn{ \frac{d}{dt}\left[-(u_t,\D u)+\frac12 \s(\|\g u\|^2)\|\D u\|^2 \right] 
+
\phi(\|\g u\|^2) \|\D u\|^2+(f'(u), |\g u|^2)
} \nonumber 
\\ & &\le 
\|\nabla  u_t\|^2+    \s'(\|\g u\|^2) (\g u,\g u_t)\|\D u\|^2 
 +\|h\| \|\D u\|.  
\end{align}
In the case when $d\ge 3$ and (\ref{f-supercrit}) holds, we have
\[
(f'(u), |\g u|^2)\ge c_0\int_\O |u|^{p-1}|\g u|^2 dx -c_1\|\g u\|^2,
~~~ c_0, c_1>0.
\]
In other (non-supercritical) cases, due to the embedding $H^1(\O)\subset L_{p+1}(\O)$, from (\ref{1st-apri})
we have the relation 
$|(f'(u), |\g u|^2)|\le  c_{R,T}\|\Delta  u\|^2$.
This implies that
\begin{align}\label{ut-Du}
\frac{d}{dt}\left[-(u_t,\D u)+
\frac12 \s(\|\g u\|^2)\|\D u\|^2 \right] 
\le 
\|\nabla u_t\|^2 +c_{R,T}(1+ \|\g u_t\|)\cdot \|\D u\|^2~+C_{R,T}.  
\end{align}
for every $t\in [0,T]$. Let 
\[
\Psi(t)=\cE(u(t),u_t(t))+\eta \left[-(u_t,\D u)+\frac12 \s(\|\g u\|^2)\|\D u\|^2 \right]
\]
with $\eta>0$. We note that there exists $\eta_*=\eta(R,T)>0$ such that
\begin{equation}\label{low-bnd}
\Psi(t)\ge \alpha_{R,T,\eta}  \left[ \|u_t\|^2+\|\D u\|^2 \right]-C_{R,T},
~~~ t\in [0,T],
\end{equation}
for every $0<\eta<\eta_*$.
Therefore using the energy relation (\ref{8.1.4}) for the approximate 
solutions  and also (\ref{ut-Du}) one can choose $\eta>0$ such that
\[
\frac{d}{dt}\Psi (t)\le c_0[\Psi(t)+c_1] (1+\| \g u_t\|^2),
 ~~~ t\in [0,T],
\]
with appropriate $c_i>0$.
By (\ref{2st-apri}) and (\ref{low-bnd})
this implies the estimate
\[
 \|u^N_t(t)\|^2+\|\D u^N(t)\|^2\le 
C_R(T) \left[1+\|\D u^N(0)\|^2 \right],~~~ t\in [0,T].
\]
\par
The above a priori estimates show that
$(u_N;\.u_N)$ is ${}^*$-weakly compact in
\[
\cW_T\equiv L_\infty(0,T; H^2(\O))\cap L_{p+1}(\O))\times \left[L_\infty  (0,T; L_2(\O))\cap 
L_2(0,T; H^1_0(\O))\right]
\quad \mbox{for every}~~ T>0.
\]
Moreover,  using the equation for $u^N(t)$ we can  show
 in the standard way that
\begin{equation}\label{3rd-apri}
\int_0^T \|\:u^N(t)\|_{-m}^2 dt \le C_T(R), \quad N=1,2,\ldots,
\end{equation}
for some $m\ge \max\{1,d/2\}$. Thus 
the Aubin-Dubinsky theorem (see \cite[Corollary 4]{Simon})
yields that $(u_N;\.u_N)$ is also  compact in
\[
C(0,T; H^{2-\e}(\O)) \times [C(0,T; H^{-\e}(\O))\cap L_2(0,T;H^{1-\e}(\O))]
~~\mbox{
for every $\e>0$.}
\]
 Thus there exists  an element  $(u;u_t)$ in $\cW_T$ such that
(along a subsequence) the following convergence holds:
\[
\max_{[0,T]}\| u^N(t)-u(t) \|_{2-\eps}^2+
\int_0^T\| u^N_t(t)-u_t(t) \|_{1-\eps}^2
dt\to 0~~\mbox{as}~~ N\to \infty.
\]
Moreover, by the Lions Lemma (see Lemma 1.3 in \cite[Chap.1]{Lio69})
we have that 
\[
f(u^N(x,t))\to f(u(x,t))~~\mbox{weakly in}~~ L_{1+1/p}([0,T]\times \O).
\]
This allows us to make a limit transition in nonlinear terms and prove the
existence of a weak solution under the additional condition
$ u_0\in (H^2\cap H^1_0)(\O)$.
One can see that this solution possesses the properties (\ref{u-tt}),
(\ref{basic-bnd}), (\ref{sm-bnd})
and satisfies the corresponding energy inequality.
\medskip\par  
Now we prove that  (\ref{dif-bnd}) (and also (\ref{dif-bnd+}) 
in the supercritical case) hold for every couple   $u^1(t)$ 
and $u^2(t)$ of weak solutions. 
For this we use the same idea as \cite{KaZe09}
and start with the following preparatory lemma which we also use
in the further considerations.
\begin{lemma}\label{le:z-mult}
Let  $u^1(t)$ 
and $u^2(t)$ be two weak solutions  to (\ref{main_eq})
with different initial data $(u^i_0;u_1^i)$ from $\cH$ such that
\begin{equation}\label{weak-bnd}
\| u_t^i(t)\|^2+ \|\g u^i(t) \|^2\le R^2~~\mbox{for all}~~t\in [0,T]
~~\mbox{and for some}~~R>0.
\end{equation}
Then for  $z(t)=u^1(t)-u^2(t)$ we have the relation
\begin{align}\label{array-z-z-eq}
\frac{d}{dt}\left[ (z,z_t)+\frac14 \s_{12}(t)
\cdot \|\g z\|^2\right]+ \frac12 \phi_{12}(t)\cdot \|\g z\|^2 + (f(u^1)-f(u^2),z)\qquad {}\qquad{}
\nonumber
\\   +\, \widetilde\phi_{12}(t) |(\g (u^1+u^2),\g z)|^2
\le  \; \|z_t\|^2+  C_R\left(\|\g u^1_t\|+\|\g u^2_t\|\right) \|\g z\|^2
\end{align}
for all $t\in [0,T]$,
where   $\s_{12}(t)= \s_{1}(t) +\s_{2}(t)$   and 
 $\phi_{12}(t)= \phi_{1}(t) +\phi_{2}(t)$
with  $\s_{i}(t)=\s(\|\g u^i(t)\|^2)$ and   
$\phi_{i}(t)=\phi(\|\g u^i(t)\|^2)$.  We also use the following notation
\begin{equation}\label{phi-tild}
\widetilde\phi_{12}(t)=\hf \int_0^1 \phi'(\la \|\g u^1(t)\|^2
+(1-\la) \|\g u^2(t)\|^2 )d\la.
\end{equation}
\end{lemma}
\begin{remark}\label{re:z-mult}
{\rm It follows directly from Definition~\ref{de:weak} that 
(\ref{u-tt}) holds for every weak solution. This and also (\ref{u-ut-d})
allows us to show that   $(z,z_t)+ \s_{12}(t)  \|\g z\|^2/4$
is absolutely continuous  with respect to $t$ and thus
the relation in (\ref{array-z-z-eq}) has a meaning for every couple
of weak solutions.  
}
\end{remark}
\begin{proof}
One can see that  $z(t)=u^1(t)-u^2(t)$ solves the equation
\begin{equation}\label{abs-dif}
    z_{tt} -\frac12\s_{12}(t)\D z_t -  \frac12\phi_{12}(t)\D z
  +G(u^1,u^2; t)=0,
\end{equation}
where 
\[
G(u^1,u^2; t)=  -\frac12 \left\{[ \s_{1}(t) -\s_{2}(t)] 
\D(u^1_t+ u^2_t) +
[  \phi_{1}(t) -\phi_{2}(t)] \D (u^1+u^2)\right\} + f(u^1)-f(u^2).
\]
Since $G\in L_2(0,T;H^{-1}(\O))+L_\infty(0,T; L_{1+1/p}(\O))$ and
$z\in L_\infty(0,T; (H^1_0\cap L_{p+1})(\O))$
for any couple  $u^1$ and $u^2$ of weak solutions, we can
multiply  equation (\ref{abs-dif})  by $z$ in $L_2(\O)$.
Therefore using  the relation
\[
|\s'_{12}(t)|
\le  C_R\left(\|\g u^1_t\|+\|\g u^2_t\|\right) 
\]
and also the observation made in Remark~\ref{re:z-mult} we conclude  that 
\begin{eqnarray*}
\frac{d}{dt}\left[ (z,z_t)+\frac14 \s_{12}(t)
\cdot \|\g z\|^2\right]+ \frac12 \phi_{12}(t)\cdot \|\g z\|^2 +
 (G(u^1,u^2,t),z)
\nonumber
\\  
\le \, \|z_t\|^2+  C_R\left(\|\g u^1_t\|+\|\g u^2_t\|\right) 
\cdot \|\g z\|^2.
\end{eqnarray*}
One can see that 
$\phi_{1}(t) -\phi_{2}(t)=
 2(\g (u^1+u^2),\g z)\cdot \widetilde{\phi}_{12}(t)$,
where $\widetilde{\phi}_{12}$ is given by (\ref{phi-tild}),
and 
\[
|[ \s_{1}(t) -\s_{2}(t)] (\g(u^1_t+ u^2_t),\g z)|
\le C_R\left(\|\g u^1_t\|+\|\g u^2_t\|\right) 
\cdot \|\g z\|^2.
\]
Thus using the structure of the term $G(u^1,  u^2;t)$ we obtain (\ref{array-z-z-eq}).
\end{proof}
\begin{lemma}\label{le:f-coerc}
Assume that $f(u)$ satisfies Assumption~\ref{A1}
and the additional requirement\footnote{This requirement holds automatically
in the supercritical case, see (\ref{f-supercrit}).}  saying that  $f'(u)\ge -c$
for some $c\ge 0$. Then for $z=u^1-u^2$ we have that
\begin{equation}\label{p-est1}
\int_\O (f(u^1)-f(u^2))  (u^1-u^2) dx\ge -c_0\|z\|^2 +c_1
\int_\O(|u^1|^{p-1}+|u^2|^{p-1}) |z|^2 dx
\end{equation}
and
\begin{equation}\label{p-est2}
\int_\O (f(u^1)-f(u^2))  (u^1-u^2) dx\ge -c_0\|z\|^2 +c_1
\int_\O|z|^{p+1} dx,
\end{equation}
where $c_0\ge 0$ and $c_1>0$ in the case when (\ref{f-supercrit}) holds and
$c_1=0$ in other cases.
\end{lemma}
\begin{proof}
It is sufficient to consider the case when (\ref{f-supercrit}) holds.
\par
The relation
in (\ref{p-est1}) follows from the obvious inequality 
\[
\int_0^1|(1-\la)u^1+\la u^2|^r d\la \ge c_r \left(|u^1|^r
+|u^2|^r\right),~~ r\ge 0,~~ u^i\in\R,
\]
which can be obtained by the direct calculation of the integral.
As for (\ref{p-est2}) we use the obvious representation 
\[
\int_{u_1}^{u_2}|\xi|^r d\xi   =\frac1{r+1}\left( |u^1|^r u^1-
|u^2|^r u^2\right),~~ r\ge 0,~~ u^i\in\R, ~~ u^1<u^2,
\]
and the argument given in \cite[Remark 3.2.9]{cl-book}.
\end{proof}

Now we return to the proof of relations (\ref{dif-bnd}) and (\ref{dif-bnd+}).
\par
Let $u^1$ and $u^2$ be weak solutions satisfying (\ref{weak-bnd})
and also the inequality $\|u^i(t)\|_{L_{p+1}(\O)}\le R$
for all $t\in [0,T]$ in the supercritical case.
We first note that in the non-supercritical case by the embedding
$H^1(\O)\subset L_{r}(\O)$ for $r=\infty$ in the case $d=1$,
 for arbitrary $1\le r<\infty$ when $d=2$ and for
 $r=2d(d-2)^{-1}$ in the case $d\ge 3$ we have that 
\begin{equation}\label{f-est-h1}
\|f(u^1)-f(u^2)\|_{-1}\le C_R \|\g (u^1-u^2)\|, ~~ u^1,u^2\in H_0^1(\O), 
~\|\g u^i\|\le R,
\end{equation}
which implies  that 
$|(f(u^1)-f(u^2),z)|\le C_R \|\g z\|^2$.
Therefore 
 it follows from Lemma \ref{le:z-mult} and from Lemma \ref{le:f-coerc}
in the supercritical case that 
\begin{align} \label{array-eq1}
 \frac{d}{dt}\left[ (z,z_t)+\frac14
\s_{12}(t) \|\g z\|^2\right] &+  \frac12\phi_{12}(t)\|\g z\|^2
 + c_0\left[
\int_\O|z|^{p+1} dx +
\int_\O(|u^1|^{p-1}+|u^2|^{p-1}) |z|^2 dx \right]
 \nonumber
 \\ 
 & \le   \, \|z_t\|^2+  C_R\left( 1+\|\g u^1_t\|+|\g u^2_t\|\right) \|\g z\|^2,
\end{align} 
where $c_0$ is positive in the supercritical case only.
\par

Now we consider the multiplier  $\cA^{-1} z_t$.
Since $H^{2-\eta}(\O)\subset L_{p+1}(\O)$ for some $\eta>0$
under the condition  $p<p_{**}=(d+4)/(d-4)_+$, 
 we easily obtain that
\begin{equation}\label{zt-Lp+1}
\|\cA^{-1} z_t\|_{L_{p+1}}^2\le C \|\cA^{-\eta/2} z_t\|^2
\le   \eps \|z_t\|^2+ C_\eps\|\cA^{-1/2} z_t\|^2~~\mbox{for every $\eps>0$}.
\end{equation}
Thus we can  multiply equation (\ref{abs-dif}) 
by $\cA^{-1} z_t$ and
 obtain that 
\begin{align}\label{ar:mult1}
 \frac12 \frac{d}{dt} \|\cA^{-1/2}z_t\|^2+
\frac12\phi_{12}(t) (z,z_t)+\hf\s_{12}(t)\| z_t\|^2 +
(G(u^1,u^2; t),\cA^{-1} z_t)=0,
\end{align}
where 
\begin{equation}\label{g-zt-min}
(G(u^1,u^2; t),\cA^{-1} z_t)=G_1(t)+G_2(t)+G_3(t).
\end{equation}
Here 
\[
G_1(t)=  -\frac12 [ \s_{1}(t) -\s_{2}(t)] (\D(u^1_t+ u^2_t),\cA^{-1} z_t), 
\]
\[
G_2(t)=  \widetilde\phi_{12}(t) 
 (\g (u^1+u^2),\g z) (\g (u^1+u^2),\g \cA^{-1} z_t)
\]
with $\widetilde\phi_{12}(t)$ given by (\ref{phi-tild}),
and
$G_3(t)= (f(u^1)-f(u^2),\cA^{-1} z_t)$.
\par 
One can see that 
$|(G_1(t) +G_2(t)|\le C_R \|z_t\|\cdot\|\g z\|$.
In the non-supercritical case by (\ref{f-est-h1}) we have the same estimate
for $|G_3(t)|$.
In the supercritical case 
we obviously have that
\begin{align} \label{f-a-1} 
\lefteqn{
\int_\O |f(u^1)-f(u^2)| |\cA^{-1} z_t| dx }\\
&&\le
 \eps  
\int_\O(1+|u^1|^{p-1}+|u^2|^{p-1}) |z|^2 dx +
C_\eps \int_\O(1+|u^1|^{p-1}+|u^2|^{p-1}) |\cA^{-1} z_t|^2 dx
\nonumber
\\ 
&&\le
 \eps  
\int_\O(1+|u^1|^{p-1}+|u^2|^{p-1}) |z|^2 dx +
C_\eps \left[\int_\O(1+|u^1|^{p+1}+|u^2|^{p+1}) dx\right]^{\frac{p-1}{p+1}}
 \|\cA^{-1} z_t\|_{L_{p+1}}^2.\nonumber 
\end{align} 
Therefore using (\ref{zt-Lp+1}) we have that
\begin{multline*}\label{}
|(G(u^1,u^2; t),\cA^{-1} z_t)| 
\le  C_R  \|z_t\|\cdot\|\g z\|  \\
  +\,
 \eps  \left[
\int_\O(|u^1|^{p-1}+|u^2|^{p-1}) |z|^2 dx + \|z\|^2+ \|z_t\|^2\right] +
C_\eps(R)\|\cA^{-1/2} z_t\|^2  
\end{multline*}
for any $\eps>0$.
Thus from (\ref{ar:mult1}) we obtain 
\begin{multline} \label{G-zt}
 \frac12 \frac{d}{dt}
 \|\cA^{-1/2}z_t\|^2+\hf \s_{12}(t)\| z_t\|^2 
  \le  C_R   \|z_t\|\cdot\|\g z\| 
  \\ 
 +\,
 \eps c_0 \left[
\int_\O(|u^1|^{p-1}+|u^2|^{p-1}) |z|^2 dx + \|z\|^2+ \|z_t\|^2\right] + c_0
C_\eps(R)\|\cA^{-1/2} z_t\|^2 
\end{multline}
for any $\eps>0$, where $c_0=0$ in the non-supercritical case.
Let 
\begin{equation}\label{psi-t-w}
\Psi(t)= \frac12  \|\cA^{-1/2}z_t\|^2+\eta \left[ (z,z_t)+\frac14
\s_{12}(t) \|\g z\|^2\right]
\end{equation}
for $\eta>0$ small enough. It is obvious that
for $\eta\le \eta_0(R)$ we have
\begin{equation}\label{eq-nrms}
 a_R\eta  \left[ \|\cA^{-1/2}z_t\|^2+
 \|\g z\|^2\right]\le \Psi(t)\le b_R  \left[ \|\cA^{-1/2}z_t\|^2+
 \|\g z\|^2\right].
\end{equation}
From  (\ref{array-eq1}) and (\ref{G-zt}) we also have that
\begin{multline*} 
 \frac{d\Psi}{dt}
 +\left[\hf \s_{12}(t)-\eta-c\eps\right]\| z_t\|^2
 +c_0\eta 
\int_\O|z|^{p+1} dx  \\+ c_0(\eta -\eps) 
\int_\O(|u^1|^{p-1}+|u^2|^{p-1}) |z|^2 dx\nonumber 
  \le C_\eps(R)\left[ \|\g z\|^2
  +
\|\cA^{-1/2} z_t\|^2\right].
\end{multline*}
After selecting appropriate $\eta$ and $\eps$
this implies the desired conclusion in (\ref{dif-bnd}) and (\ref{dif-bnd+}).
\par
We can use   (\ref{dif-bnd}) and (\ref{dif-bnd+}) to prove the existence 
 of weak solutions for initial data
 $(u_0; u_1) \in \cH$  by limit transition from  smoother solutions.
Indeed, we can choose a sequence $(u_0^n;u_1^n)$ elements 
from $(H^2\cap H^1_0)(\O)\times L_2(\O)$
such that $(u_0^n;u_1^n)\to (u_0;u_1)$ in $\cH$.
Due to (\ref{dif-bnd}) and (\ref{dif-bnd+}) 
the corresponding solutions $u^n(t)$ 
converge to a function $u(t)$ in the sense that
\[
\max_{t\in [0,T]}\left\{
\|u^n_t(t)-u_t(t)\|^2_{-1}+ \| u^n(t)-u(t)\|_1^2\right\} +\int_0^T\|u^n(\t)-u(\t)\|^{p+1}_{L_{p+1}(\O)} d\t \to 0.
\]
From the boundedness provided by the energy relation in (\ref{basic-bnd})
for $u^n$
 we also have $*$-weak convergence of $(u^n;u_t^n)$  to  $(u;u_t)$
in the space  
\[
 L_\infty(0,T; H^1(\O))\cap L_{p+1}(\O))\times \left[L_\infty  (0,T; L_2(\O))\cap 
L_2(0,T; H^1_0(\O))\right].
\] This implies that $u(t)$ is a weak solution.
By (\ref{dif-bnd}) this solution is unique.
Moreover, this solution satisfies the corresponding energy 
{\em inequality}.
\medskip\par
Now we prove smoothness properties of weak solutions
stated in (\ref{smoth-prop}) using the same method as 
\cite{KaZe09} (see also \cite{BabinVishik}).
\par
As usual the argument below can be justified by considering 
Galerkin approximations.
\par 
Let $u(t)$ be a solution such that $\|(u(t);u_t(t))\|_\cH\le R$
for $t\in [0,T]$.
Formal differentiation gives that $v=u_t(t)$
solves the equation
\begin{equation}\label{abs-dif-d}
    v_{tt} -\s(\|\g u\|^2)\D v_t -  \phi(\|\g u\|^2)\D v
  +f'(u)v + G_*(u,u_t; t)=0,
\end{equation}
where 
\[
G_*(u,u_t; t)=  -2\left[\s'(\|\g u\|^2) \D u_t +
  \phi'(\|\g u\|^2)\D u  \right] (\g u,\g u_t).
\]
Thus, multiplying  equation (\ref{abs-dif-d})  by $v$ we have that 
\begin{eqnarray*}
 \frac{d}{dt}\left[ (v,v_t)+\frac12
\s(\|\g u\|^2) \|\g v\|^2\right]+  \phi(\|\g u\|^2)\|\g v\|^2+(f'(u)v,v) 
  \\
 \le \, \|v_t\|^2+  C_R\left[ | (\g u,\g v)|^2  + | (\g u,\g v)| 
\|\g v\|^2\right].
\end{eqnarray*}
This implies that 
\[
 \frac{d}{dt}\left[ (v,v_t)+\frac12
\s(\|\g u\|^2) \|\g v\|^2\right]+  
c_0\int_\O |u|^{p-1}v^2 dx
 \le \, \|v_t\|^2+  C_R\left[ 1  + \|\g u_t\| \right]
\|\g v\|^2,
\]
where $c_0>0$ in the supercritical case only.
Using the multiplier $\cA^{-1}v_t$ in (\ref{abs-dif-d}) we obtain that
\begin{eqnarray*}
 \frac12 \frac{d}{dt} \|\cA^{-1/2}v_t\|^2+\s(\|\g u\|^2)\| v_t\|^2
 \le 
C_R\|\g v\|\|v_t\|+
\left|\int_\O f'(u)v\cA^{-1} v_t dx \right|.
\end{eqnarray*}
As above (cf.\ \eqref{f-a-1}) in the supercritical case
  we have that 
\[
\left|\int_\O f'(u)v\cA^{-1} v_t dx \right| \le
 \eps  
\int_\O(1+|u|^{p-1}) |v|^2 dx +
C_{R,\eps}
 \|\cA^{-1} v_t\|_{L_{p+1}}^2.
\]
for any $\eps>0$. Thus
\begin{eqnarray*}
\lefteqn{ \frac12 \frac{d}{dt} \|\cA^{-1/2}v_t\|^2+ \s(\|\g u\|^2)\| v_t\|^2} \\
& & \le \, \eps\left( \| v_t\|^2+c_0\int_\O |u|^{p-1}v^2 dx
\right) +
C_{R,\eps}\left[\|\g v\|^2 + 
 \|\cA^{-1} v_t\|_{L_{p+1}}^2\right].
\end{eqnarray*}
We introduce now the functional 
\[
\Psi_*(t)= \frac12  \|\cA^{-1/2}v_t\|^2+\eta \left[ (v,v_t)+\frac12
\s(\|\g u\|^2) \|\g v\|^2\right]
\]
for $\eta>0$ small enough. It is obvious that
for $\eta\le \eta_0(R)$ we have
\[
 a_R\eta  \left[ \|\cA^{-1/2}v_t\|^2+
 \|\g v\|^2\right]\le \Psi_*(t)\le b_R  \left[ \|\cA^{-1/2}v_t\|^2+
 \|\g v\|^2\right].
\]
Using (\ref{zt-Lp+1})
we also have that
\begin{align*} 
 \frac{d\Psi_*}{dt}
 +\left[\s(\|\g u\|^2)-\eta-\eps\right]\| v_t\|^2
 +c_0[\eta-\eps] \int_\O|u|^{p-1} v^2 dx 
\nonumber  
\\   \le C_{R,\eps}\left( 1+\|\g u_t\|^2\right)\left[  \|\cA^{-1/2}v_t\|^2+
 \|\g v\|^2\right].
\end{align*}
In particular for $\eta>0$ small enough, there exists $\a_R>0$ such  that 
\begin{equation}\label{g-vt} 
 \frac{d\Psi_*}{dt}+\alpha_R\left( \| v_t\|^2
 +c_0 \int_\O|u|^{p-1} v^2 dx\right)
  \le  C_R\left( 1+\|\g u_t\|^2\right) \Psi_*(t),
\end{equation}
where $c_0>0$ in the supercritical case only.
This implies that
\begin{align*}
\| u_{tt}(t)\|_{-1}^2+ \| \g u_{t}(t)\|^2 
+\! \int_0^{t}\!\!\left[\|u_{tt}(\t)\|^2+
c_0\int_\O|u(x,\t)|^{p-1}|u_t(x,\t)|^2dx\right]  d\t\nonumber
\\
\le C_{R,T} \left(\| u_{tt}(0)\|_{-1}^2+ \| \g u_t(0)\|^2\right)
\end{align*}
for $t\in [0,T]$,
where  $c_0=0$ in the non-supercritical case. 
This formula demonstrates preservation of some smoothness.
To obtain (\ref{smoth-prop}) we multiply (\ref{g-vt}) by $t^\alpha$.
This gives us the relation
\begin{equation}\label{g-vt-1} 
 \frac{d}{dt}
(t^\a\Psi_*)+\alpha_R t^\a \| v_t\|^2
  \le C_R\left( 1+\|\g u_t\|^2\right)[ t^\a \Psi_*]
+\a t^{\a-1} b_R\left[  \|\cA^{-1/2}v_t\|^2+
 \|\g v\|^2\right].
\end{equation}
One can see that 
\[
t^{\a-1}\|\g v\|^2  \le 1+ t^{2(\a-1)} \|\g u_t\|^2\|\g v\|^2\le 
C_T[1+\|\g u_t\|^2(t^\a\Psi_*)],~~ t\in [0,T],
\]
provided $\alpha\ge 2$. We also have that
$\|\cA^{-1/2} v_t\|^2  \le C \|v_t\|^\delta 
\|\cA^{-m} u_{tt}\|^{2-\delta}$ for any  $m\ge1 $ 
with $\delta=\delta(m)\in [1,2)$.
Since 
\[
\cA^{-m}u_{tt} =\s(\|\g u\|^2)\cA^{-m+1}u_t +
 \phi(\|\g u\|^2) \cA^{-m+1} u - \cA^{-m}( f(u)-h),
\]
one can see that
$\|\cA^{-m} u_{tt}\|\le C_R + \int_\O |f(u)|dx\le \tilde C_R$ 
for $m\ge \max\{1,d/2\}$.
Therefore
\[
t^{\a-1}\|\cA^{-1/2} v_t\|^2  \le C_{\delta} t^{(\a-1)} \|v_t\|^\delta 
\le \eps  t^{\a} \|v_t\|^2 +  C_{R,T,\delta,\eps},~~ t\in [0,T],
\]
provided $2(\alpha-1)/\delta\ge \alpha$.
Thus from (\ref{g-vt-1}) we have that 
 \[
 \frac{d}{dt}
(t^\a\Psi_*)   
  \le C_{R,T}+ C_{R,T}\left( 1+\|\g u_t\|^2\right)[ t^\a \Psi_*].
\]
This implies (\ref{smoth-prop}) with some $\beta>0$.
\par 
Now we prove that the function $t\mapsto (u(t);u_t(t))$ 
is (strongly) continuous  in $\cH=[H_0^1(\O)\cap L_{p+1}(\O)]\times L_2(\O)$
and establish energy relation (\ref{8.1.4}).
We concentrate on the supercritical case only (other cases are much simpler).
\par 
We first note that
the function $t\mapsto (u(t);u_t(t))$ is weakly continuous in $\cH$
for every $t\ge 0$ and $t\mapsto u(t)$ is strongly continuous in $H^1_0(\O)$,
$t\ge 0$. Moreover,
  (\ref{smoth-prop}) implies that 
$t\mapsto (u(t);u_t(t))$ is continuous in 
$H_0^1(\O)\times L_2(\O)$ at every point $t_0>0$. 
\par 
Let us prove that $t\mapsto \|u(t)\|^{p+1}_{L_{p+1}(\O)}$ is continuous 
at $t_0>0$.
From (\ref{smoth-prop}) and from the energy inequality 
for weak solutions we have that 
\begin{equation}\label{fin-int}
\int_a^b\int_\O |u|^{p-1}(|u|^{2}+ |u_t|^2) dx dt \le C_{a,b},~~~ \mbox{for all}~~ 
0<a<b\le T.
\end{equation}
On  smooth functions we also have that
\[
\left|\frac{d}{dt}\|u(t)\|^{p+1}_{L_{p+1}(\O)}\right|=(p+1)
\left|\int_\O |u|^p u_t dx\right| 
\le\frac{p+1}2\int_\O |u|^{p-1}(|u|^{2}+ |u_t|^2) dx. 
\]
Therefore by (\ref{fin-int}) for $t_2>t_1>a$ we have that 
 \[
\left|\|u(t_2)\|^{p+1}_{L_{p+1}(\O)} -\|u(t_1)\|^{p+1}_{L_{p+1}(\O)} \right| 
\le\frac{p+1}2\int_{t_1}^{t_2}\int_\O |u|^{p-1}(|u|^{2}+ |u_t|^2) dxdt 
\to 0~~\mbox{as}~~t_2-t_1\to 0.
\]
Thus the function $t\mapsto \|u(t)\|^{p+1}_{L_{p+1}(\O)}$ is continuous
for $t>0$. Since $u(t)$ is weakly continuous in $L_{p+1}(\O)$ for $t>0$
and $L_{p+1}(\O)$ is uniformly convex, we conclude 
that $u(t)$ is norm-continuous in  $L_{p+1}(\O)$ at every point $t_0>0$.
\par 
In the next step we establish energy relation (\ref{8.1.4}) for every
 $t>s>0$. For this we note that by (\ref{smoth-prop})
equation (\ref{main_eq}) is satisfied on any interval 
$[a,b]$, $0<a<b\le T$,
as an equality in  space $\left[H^{-1}+L_{1+1/p}\right](\O)$. Moreover one can see that 
$f(u)u_t\in L_1([a,b]\times\O)$. This allows us  
to multiply equation (\ref{main_eq}) by $u_t$ and prove (\ref{8.1.4}) 
for $t\ge s>0$.
\par 
To prove energy relation  (\ref{8.1.4}) for $s=0$ we note that
it follows from  (\ref{8.1.4}) 
for $t\ge s>0$ that the limit $\cE(u(s),u_t(s))$
 as $s\to 0$
exists    and 
\[
\cE_*\equiv \lim_{s\to 0}\cE(u(s),u_t(s))
=\cE(u(t), u_t(t))+\int_0^t \s(\|\g u(\tau)\|^2) \|\g u_t(\tau)\|^2 d\tau.
\]
Since  $u(t)$
is continuous in $H^1_0(\O)$ on $[0,+\infty)$,
we conclude that there is a sequence $\{s_n\}$, $s_n\to 0$,
such that $u(x,s_n)\to u_0(x)$ almost surely. Since 
$F(u)\ge -c$
for all $u\in \R$, from Fatou's lemma we have that
\[
\int_\O  F(u_0(x))dx\le \liminf_{s\to 0}\int_\O  F(u(x,s))dx. 
\]
The property of weak continuity of $u_t(t)$ at zero implies that
$\|u_1\|^2\le  \liminf_{s\to 0}\|u_t(s)\|^2$.
Thus we arrive to the relation $\cE(u_0,u_1)\le \cE_*$.
 Therefore from the energy 
{\it inequality}  for weak solutions we obtain 
 (\ref{8.1.4})  for all $t\ge s\ge 0$.
\par
No we conclude the proof of strong continuity of 
$t\mapsto (u(t);u_t(t))$ in $\cH$ at $t=0$.
From the continuity of $t\mapsto\cE(u(t), \.u(t))$ and property
that $u(t)\to u_0$ in $H^1_0(\O)$ as $t\to 0$ one can see by contradiction
 that 
\[ \lim_{t\to 0}\|u_t(t)\|^2= \|u_1\|^2,~~~
\lim_{t\to 0}\int_\O  F(u(x,t))dx =
\int_\O  F(u_0(x))dx.
\]
The first relation implies that  $u(t)$ is continuous in $L_2(\O)$
at $t=0$. It follows from Assumption~\ref{A1}  that
\[
|u(x,t)|^{p+1}\le C_1  F(u(x,t)) +C_2~~\mbox{for almost all} ~~ x\in \O,~ t>0.
\]
We also have that $|u(x,t)|^{p+1}\to |u_0(x)|^{p+1}$ almost everywhere
along some sequence as $t\to 0$. Therefore from the Lebesgue 
dominated convergence theorem we conclude that
\[
 \|u(t)\|^{p+1}_{L_{p+1}(\O)} \to \|u_0\|^{p+1}_{L_{p+1}(\O)} 
~~\mbox{as}~~t\to 0
\]
along a subsequence. Using again  uniform convexity 
of the  space $L_{p+1}(\O)$ we conclude that 
$u(t)$ is strongly continuous in $L_{p+1}(\O)$.
The proof of Theorem~\ref{pr:wp1} is complete.
\end{proof}

\begin{remark}\label{re-mult}
{\rm
We do not know how to avoid the assumption
 $p<p_{**}=(d+4)/(d-4)_+$ (which arises in dimension $d$
greater than 4) in the proof of well-posedness. 
The point is that we cannot use smother 
multipliers like $\cA^{-2l}z_t$ and  $\cA^{-2l}z$
to achieve the goal
because the term $\|\g z\|^2$ goes into picture in the estimate for $G$.
If we will use the multipliers  $\cA^{-2l}z_t$ and  $z$
in the proof of uniqueness of solutions,
then we get a problem with the corresponding two-sided 
estimate for the corresponding analog of the
 function $\Psi(t)$ given by (\ref{psi-t-w}).
\par 
As for the existence of weak solutions  without the requirement 
$p\ge p_{**}$ in the case 
$d\ge 4$   we note that  
the  standard a priori   estimates for $u^N(t)$
(see (\ref{1st-apri}), (\ref{2st-apri}) and (\ref{3rd-apri})) 
can be also
easily obtained in this case. 
The main difficulty in 
this situation is the limit transition in the nonlocal terms 
$\phi(\|u^N(t)\|^2)$ and  $\sigma(\|u^N(t)\|^2)$. To do this we 
can apply the same procedure as in \cite{CCFS98} with
$\sigma=const$, $f(u)\equiv 0$. We do not provide details because
we do not know how establish uniqueness for this case. 
}
\end{remark}
\begin{remark}\label{re:e+}
{\rm In addition to Assumption~\ref{A1}
assume that either  
\begin{equation}\label{Phi-coer}
\Phi(s) \equiv \int_0^s \phi(\xi) d\xi \to+\infty~~\mbox{as}~~s\to+\infty
~~~\mbox{and}~~~ \mu_f>0, 
\end{equation}
or else
\begin{equation}\label{44a}
\hat\mu_\phi:= \liminf_{s\to+\infty} \phi(s)>0~~~\mbox{and}~~
 ~\hat\mu_\phi\la_1 +\mu_f>0,
\end{equation}
where $\mu_f$ is defined by (\ref{f-coercive}) and
$\la_1$ is the first eigenvalue of  the minus Laplace operator
 in $\O$ with the Dirichlet boundary conditions
(if  $\hat\mu_\phi=+\infty$, then $\mu_f>-\infty$ can be arbitrary). 
In this case it
is easy to see that  (\ref{e+w-eta}) holds with $\eta=\nu=0$.
Therefore  the energy relation in  (\ref{8.1.4})
yields
\begin{equation}\label{R-bnd-1}
\sup_{t\in\R_+}\cE^0_+(u(t), u_t(t))\le C_R~~
\mbox{provided}~~ \cE^0_+(u_0, u_1)\le R,
\end{equation}
where $R>0$ is  arbitrary
and  $\cE^0_+$ is defined by (\ref{e+eta}) with $\eta=0$. 
Now using either (\ref{Phi-coer}) of (\ref{44a})
 we can conclude from (\ref{R-bnd-1}) that 
\begin{equation}\label{nabla-sigma}
\sup_{t\in\R_+}\|\g u(t)\|\le C_R~~~\mbox{and}~~~ 
\inf_{t\in\R_+}\s(\|\g u(t)\|^2)\ge \s_R>0.
\end{equation}
 Therefore under the conditions above
the energy relation in (\ref{8.1.4}) along with (\ref{R-bnd-1})  implies that 
\begin{equation}\label{R-bnd}
\sup_{t\in\R_+}\cE^0_+(u(t), u_t(t))+\int_0^\infty  \|\g u_t(\tau)\|^2 d\tau
\le C_R
\end{equation}
for any initial data such that  $\cE^0_+(u_0, u_1)\le R$. 
We note  that in the case considered  
the energy type function $\cE^0_+$ is topologically equivalent
to the norm on $\cH$ in the sense that 
$\cE^0_+(u_0, u_1)\le R$ for some $R>0$ if and only if 
$\|(u_0; u_1)\|_\cH\le R_*$ for some $R_*>0$.
}
\end{remark}

\section{Long-time dynamics}
\label{Sect3}

\subsection{Generation of an evolution semigroup}
By Theorem \ref{pr:wp1} problem \eqref{main_eq} generates an
evolution semigroup   $S(t)$ in
the space $\cH$   by  the formula
\begin{equation}\label{evol-sgr}
S(t)y=(u(t);\.u(t)),~ 
\mbox{where $y=(u_0;u_1)\in\cH$ and $u(t)$ solves (\ref{main_eq})}
\end{equation}
To describe continuity properties of  $S(t)$ 
it is convenient to introduce the following notion.
\begin{definition}[Partially strong topology]\label{de:ss-top}
{\rm
A sequence  $\{(u^n_0;u^n_1)\}\subset \cH$ is said to be 
{\em partially strongly 
convergent} to $(u_0;u_1)\in \cH$ if 
$u^n_0\to u_0$ strongly in $H^1_0(\O)$, $u^n_0\to u_0$ 
weakly in $L_{p+1}(\O)$ and $u^n_1\to u_1$ strongly in $L_{2}(\O)$
as $n\to\infty$ (in the case when $d\le 2$ we take $1<p<\infty$ arbirtary). 
}
\end{definition}
\par 
It is obvious that the partially strong convergence becomes strong
in the non-supercritical case ($H_0^1(\O)\subset L_{p+1}(\O)$).
\begin{proposition}\label{pr:gener}
Let Assumption~\ref{A1} be in force. Then the evolution semigroup 
 $S(t)$ given   by  (\ref{evol-sgr}) is a continuous mapping
in $\cH$ \wrt the strong topology. Moreover,
\begin{enumerate}
    \item[{\bf (A)}] {\bf General case: } For every $t>0$ $S(t)$ maps
$\cH$ into itself continuously in the partially strong topology.
 \item[{\bf (B)}]  {\bf Non-supercritical case ((\ref{f-supercrit}) fails):}
For any $R>0$ and $T>0$ there exists $a_{R,T}>0$ such that 
\begin{equation*}
\| S(t)y_1-S(t)y_2\|_\cH \le a_{R,T}\| y_1-y_2\|_\cH,\quad t\in [0,T],
\end{equation*}
for all $y_1,y_2\in\cH=H^1_0(\O)\times L_2(\O)$ such that $\| y_i\|\le R$.
Thus, in this case $S(t)$ is a  locally Lipschitz continuous mapping in $\cH$ 
with respect to the strong topology.
\end{enumerate}
\end{proposition}
\begin{proof}
Let $(u^n_0;u^n_1)\to (u_0;u_1)$ in $\cH$ as $n\to\infty$.
From the energy relation we have that 
\begin{multline}\label{en-lin}
\lim_{n\to\infty}\left[\cE(u^n(t), u^n_t(t))+
\int_0^t \s(\|\g u^n(\tau)\|^2) \|\g u^n_t(\tau)\|^2 d\tau\right]=
\lim_{n\to\infty}\cE(u^n_0, u^n_1) \\
= \cE(u_0, u_1)= \cE(u(t), u_t(t))+
\int_0^t \s(\|\g u(\tau)\|^2) \|\g u_t(\tau)\|^2 d\tau,
\end{multline}
where  $u^n(t)$ and $u(t)$ are weak solutions 
with initial data  $(u^n_0;u^n_1)$ and $(u_0;u_1)$. 
Using (\ref{dif-bnd}) and  the low continuity 
property of  weak convergence one can see from (\ref{en-lin})
that $u^n(t)\to u(t)$ in $H^1_0(\O)$ and also
\[
\lim_{n\to\infty}\left[
\frac12 \|u^n_t(t)\|^2 + \int_\O F(u^n(x,t))dx\right]=
\frac12 \|u_t(t)\|^2 + \int_\O F(u(x,t))dx.
\]
As in the proof of the strong time continuity of weak solutions 
in Theorem~\ref{pr:wp1} this allows us to obtain the 
strong continuity \wrt initial data. 
\par 
Now we establish additional continuity properties  stated in {\bf (A)} and {\bf (B)}.
\par 
{\bf (A)} This easily follows from uniform boundedness of  
$\|u^n_t(t)\|$ and $\|u^n(t)\|_{L_{p+1}(\O)}$ on each interval $[0,T]$ 
(which implies the corresponding weak compactness) and from Lipschitz 
type estimate in (\ref{dif-bnd}) for the difference of two solutions.
We also use the fact that by (\ref{smoth-prop}) $\|\g u^n_t(t)\|$
is uniformly bounded for each $t>0$.
\par 
{\bf (B)} Let $S(t)y_i=(u^i(t);u^i_t(t))$, $i=1,2$. Then in the 
non-supercritical 
case we have (\ref{f-est-h1}). Therefore using (\ref{basic-bnd}) and 
Lemma~\ref{le:z-mult} we obtain that
\begin{equation*}
 \frac{d}{dt}\left[ (z,z_t)+\frac14
\s_{12}(t) \|\g z\|^2\right]
\le   \|z_t\|^2+  C_{R,T}\left( 1+\|\g u^1_t\|+|\g u^2_t\|\right) \|\g z\|^2,
\end{equation*}
where $z=u^1-u^2$ and $\s_{12}(t)$ is defined in Lemma~\ref{le:z-mult}.
\par 
In the case considered we can
 multiply equation (\ref{abs-dif}) 
by $z_t$
and obtain that 
\begin{equation}\label{ener-dif}
 \frac12 \frac{d}{dt} \| z_t\|^2+\hf\s_{12}(t)\|\g z_t\|^2 +
G(t) = -\hf\phi_{12}(t)(\g z, \g z_t) 
 \le  C_{R,T} \|\g z_t\| \|\g z\|
\end{equation}
Here above
\begin{equation}\label{G-zt-new}
G(t)\equiv (G(u^1,u^2; t), z_t)=H_1(t)+H_2(t)+H_3(t),
\end{equation}
where 
\[
H_1(t)=  \frac12 [ \s_{1}(t) -\s_{2}(t)] (\g (u^1_t+ u^2_t),\g z_t), 
\]
\[
H_2(t)=  \widetilde\phi_{12}(t) 
 (\g (u^1+u^2),\g z) (\g (u^1+u^2),\g  z_t)
\]
with $\widetilde\phi_{12}(t)$ given by (\ref{phi-tild}),
and $H_3(t)= (f(u^1)-f(u^2), z_t)$.
Using these representations one can see that
\begin{eqnarray*}
|(G(u^1,u^2; t), z_t)| &\le &
  C_{R,T} (1+ \|\g u^1_t\|+\| \g u^2_t\|)\|\g z_t\| \|\g z\| \\
 &\le &
  \eps\|\g z_t\|^2+ C_{R,T,\eps} (1+ \|\g u^1_t\|^2+\| \g u^2_t\|^2)\|\g z\|^2. 
\end{eqnarray*}
for any $\eps>0$. Therefore the function 
\[
V(t)= \frac12  \| z_t\|^2+
\eta \left[ (z,z_t)+\frac14
\s_{12}(t) \|\g z\|^2\right]
\]
for $\eta>0$ small enough satisfies  the relations
\[
 a_{R,T} \left[ \|z_t\|^2+
 \|\g z\|^2\right]\le V(t)\le b_{R,T}  \left[ \|z_t\|^2+
 \|\g z\|^2\right]
\]
and
\[
\frac{d}{dt}V(t)\le c_{R,T} (1+ \|\g u^1_t\|^2+\| \g u^2_t\|^2) V(t) 
\]
with positive  constants $a_{R,T}$, $b_{R,T}$ and $c_{R,T}$.
Thus Gronwall's lemma and the finiteness of the dissipation integral in
(\ref{basic-bnd}) imply the  desired conclusion.
\end{proof}

\begin{remark}\label{re:gradient} {\rm 
One can see from the energy relation in \eqref{8.1.4}
that the dynamical  system  generated by semigroup 
$S(t)$ is gradient  on $\cH$ (with respect to the strong topology), i.e.,  
 there exists a continuous functional $\Psi(y)$ on $\cH$ 
(called a \textit{strict Lyapunov function})  possessing  the properties  
(i) $\Psi\big(S(t)y\big) \leq \Psi(y)$  for all $t\geq 0$ and $y\in \cH$; 
(ii) equality $\Psi(y)=\Psi(S(t)y)$ may take place for all $t>0$ if only $y$
is a stationary point of $S(t)$.
In our case  the full energy $\cE(u_0;u_1)$ is a strict Lyapunov function.
}
\end{remark}
\subsection{Dissipativity}
Now we establish some dissipativity properties
of the semigroup $S(t)$.
Fro this we need the following hypothesis.
\begin{assumption}\label{A:dis}
{\rm 
We assume\footnote{
Under these additional
conditions the properties in (\ref{phi-si-coer}) and (\ref{phi-s-si})
holds automatically with $\eta_0=c_1=0$.
} that either (\ref{44a}) holds or else
\begin{equation}\label{phi-s}
 \phi(s)s\to +\infty~~ as ~~ s
\to +\infty~~~\mbox{and}~~~\mu_f=\liminf_{|s|\to\infty}\left\{ s^{-1}f(s)\right\}>0. 
\end{equation}
}
\end{assumption}
\begin{proposition}\label{pr:dis}
Let Assumptions~\ref{A1} and \ref{A:dis} be in force. 
Then 
there exists $R_*>0$ such that for any $R>0$ we can find
$t_R\ge 0$ such that 
\begin{equation*}
 \|(u(t); u_t(t))\|_\cH\le R_*~~\mbox{for all}~~ t\ge t_R,  
\end{equation*}
where $u(t)$ is a solution to (\ref{main_eq})
with  initial data $(u_0;u_1)\in\cH$
such that  $\|(u_0;u_1)\|_\cH\le R$.
In particular, the evolution semigroup $S(t)$ is 
dissipative in $\cH$ and
\begin{equation}\label{bs-abs}
\mathscr{B}_*=\left\{ (u_0;u_1)\in\cH\, :\; \|(u_0;u_1)\|_\cH\le R_*\right\}
~~~is~ an~ absorbing~ set.  
\end{equation}
\end{proposition}
\begin{proof}
Let
$u(t)$ be a solution to (\ref{main_eq})
with  initial data possessing the property  $\|(u_0;u_1)\|_\cH\le R$.
Multiplying   equation (\ref{main_eq}) by $u$ we obtain that  
\begin{eqnarray*}
 \frac{d}{dt}\left[ (u,u_t)+\frac12
\Sigma(\|\g u\|^2)\right] -\| u_t\|^2+  \phi(\|\g u\|^2)\|\g u\|^2+
(f(u), u)-(h,u) =0, 
\end{eqnarray*}
where
$\Sigma(s)=\int_0^s\s(\xi)d\xi$.
Therefore using the energy relation in (\ref{8.1.4}) for
the function 
$W(t)=\cW^{\eta,0}(u(t),u_t(t))$ with $\cW^{\eta,\nu}$ given by (\ref{w+eta})
we obtain that 
\begin{eqnarray*}
 \frac{d}{dt}W(t)+
\sigma(\|\g u\|^2) \|\g u_t\|^2 -\eta \| u_t\|^2 +
\eta  \phi(\|\g u\|^2)\|\g u\|^2+
\eta (f(u), u)-\eta (h,u) =0. 
\end{eqnarray*}
Since (\ref{phi-s}) implies (\ref{Phi-coer}), we 
have (\ref{nabla-sigma}).
By (\ref{f-coercive}) and (\ref{f-supercrit}) we have that
\[
(u,f(u))\ge d_0\|u\|^{p+1}_{L_{p+1}(\O)}
+d_1(\mu_f-\delta)\|u\|^2 -d_2(\delta),~~ 
\forall\, \delta>0,
\]
where $d_0>0$, $d_1=0$ in the supercritical case
and  $d_0=0$, $d_1=1$ in other cases.
In both cases (either (\ref{44a}) or (\ref{phi-s})) this yields 
\begin{eqnarray*}
 \frac{d}{dt}W(t)+
(\sigma_R -\eta) \| u_t\|^2 +
\eta c_0 \phi(\|\g u\|^2)\|\g u\|^2+
\frac{\eta d_0}{2} \|u\|^{p+1}_{L_{p+1}(\O)} +\eta c_1\|u\|^2\le \eta c_2 
\end{eqnarray*}
with positive $c_i$ independent of $R$ and $d_0>0$ 
in the supercritical case only.
Thus  
there exist constants $a_0, a_1>0$ independent of $R$ and
also $0<\eta_R\le 1$ such that
\begin{equation*}
\frac{d}{dt}\cW^{\eta,0}(u(t),u_t(t))+
\eta a_0 \left[ \| u_t\|^2 +
 \phi(\|\g u\|^2)\|\g u\|^2+
 d_0 \|u\|^{p+1}_{L_{p+1}(\O)} + \|u\|^2\right] \le \eta a_1, 
\end{equation*}
for all initial data $(u_0;u_1)\in\cH$
such that   $\|(u_0;u_1)\|_\cH\le R$ and for each $0<\eta\le \eta_R$.
Moreover, for this choice of $\eta$ we  have relation (\ref{e+w-eta})
with $\nu=0$ and  $a(\eta)\ge a(0)$. 
Therefore using the "barier" method (see, e.g.,
 \cite[Theorem 1.4.1]{Chueshov}  and \cite[Theorem 2.1]{Kolbasin})
 we can  conclude the proof.
\end{proof}
\begin{remark}\label{re:abs-inv}
{\rm Let  
$\sB_0=\left[\bigcup_{t\ge 1+ t_*}S(t)\sB_* \right]_{ps}$, where 
$\sB_*$ is given by (\ref{bs-abs}), $t_*\ge0$ is chosen such that  
$S(t)\sB_*\subset \sB_*$ for $t\ge t_*$ and $[\cdot]_{ps}$ denotes 
the closure in the partially strong topology. By the standard argument 
(see, e.g., \cite{Temam}) one can see that $\sB_0$ is a 
closed forward invariant bounded absorbing set which lies in $\sB_*$.
Moreover, by (\ref{smoth-prop}) the set $\sB_0$ is bounded
in $H_0^1(\O)\times H^1_0(\O)$.
}
\end{remark}
For  a strictly positive stiffness coefficient we can also prove 
a dissipativity property in the space 
$\cH_*= (H^2\cap H^1_0)(\O)\times L_2(\O)$\footnote{We note
that $\cH_*\subset\cH$ because $H^2(\O)\subset L_{p+1}(\O)$ 
for $p<p_{**}$.}.
Indeed, we have the following assertion.
\begin{proposition}\label{pr:dis-h2}
In addition to  the hypotheses of Proposition~\ref{pr:dis}
we assume that $\phi(s)$ is strictly positive
(i.e., $\phi(s)\ge\phi_0>0$  for all $s\in\R_+$) 
and $f'(s)\ge -c$ for all $s\in\R$ in the case when (\ref{f-crit})
holds with $p=p_*$.
Let $u(t)$ be a solution to (\ref{main_eq})
with  initial data $(u_0;u_1)\in\cH$
such that $u_0\in  H^2(\O)$ and
 $\|(u_0;u_1)\|_\cH\le R$ for some $R$.
Then 
there exist $B>0$ and $\gamma>0$ independent of $R$
and $C_R>0$  such that  
\begin{equation}\label{dis-h2}
 \|\Delta u(t)\|^2\le C_R  (1+\|\Delta u_0\|^2)
e^{-\gamma t} +B ~~\mbox{for all}~~ t\ge 0.
\end{equation}
\end{proposition}
\begin{proof}
By Proposition~\ref{pr:dis} we have that  $\|(u(t);u_t(t)\|\le R_*$
fo all $t\ge t_R$. Therefore it follows from (\ref{20a}) that
\[
\frac{d}{dt}\chi(t)
+
\frac{\phi_0}2 \|\D u(t)\|^2\le 
\|\nabla  u_t(t)\|^2+    C_{R_*}\|\g u_t(t)\|^2\|\D u(t)\|^2 
 +C_{R_*} ~~\mbox{for all}~~ t\ge t_R,  
\]
where 
$\chi(t)=-(u_t(t),\D u(t))+ \s(\|\g u(t)\|^2)\|\D u(t)\|^2/2$.
One can see that
\begin{equation}\label{hi-est}
a_1 \|\D u(t)\|^2 -a_2 \le \chi(t)
\le  a_3 \|\D u(t)\|^2 +a_4 ~~\mbox{for all}~~ t\ge t_R,  
\end{equation}
where $a_i=a_i(R_*)$ are positive constants.
Therefore using the finiteness of the dissipation integral
$\int_{t_R}^\infty\|\nabla  u_t(t)\|^2dt<C_{R_*}$ we can conclude that
\[
\chi(t) \le C_R  |\chi(t_R)|
e^{-\gamma (t-t_R)} +C_{R_*} ~~\mbox{for all}~~ t\ge t_R.
\]
Thus (\ref{hi-est}) and (\ref{sm-bnd}) yield (\ref{dis-h2}).
\end{proof}
\begin{remark}\label{re:w-attr}
{\rm
Using (\ref{sm-bnd}) one can show that the evolution operator
$S(t)$ generated by (\ref{main_eq}) maps the space
$\cH_*= (H^2\cap H^1_0)(\O)\times L^2(\O)$ 
into itself  and weakly continuous
\wrt  $t$ and initial data.
Therefore under the hypotheses of Proposition~\ref{pr:dis-h2} 
by \cite[Theorem~1, Sect.II.2]{BabinVishik}  $S(t)$
possesses a weak global attractor in $\cH_*$. Unfortunately we cannot derive from 
Proposition~\ref{pr:dis}  a similar result in the space $\cH$
because we cannot prove that $S(t)$
is a weakly closed mapping in $\cH$ 
(a mapping $S :\cH\mapsto\cH$ is said to be weakly
closed if  weak convergences $u_n\to u$ and $Su_n\to v$
imply $Su=v$). Below we prove the existence of a global attractor
in $\cH$ under additional hypotheses concerning the 
stiffness coefficient $\phi$.
}
\end{remark}

\subsection{Asymptotic compactness}
In this section we prove several properties of asymptotic 
compactness of the semigroup $S(t)$. 
\par
We start with the following 
theorem.

\begin{theorem}\label{th:ak}
Let Assumptions~\ref{A1} and \ref{A:dis} be in force.
Assume also that $\phi(s)$ is strictly positive
(i.e., $\phi(s) >0$  for all $s\in\R_+$) 
and $f'(s)\ge -c$ for all $s\in\R$ in 
the non-supercritical case (the bounds in (\ref{f-supercrit})
are not valid). 
Then there exists a bounded set $\sK$ 
in the space $\cH_1=(H^2\cap H^1_0)(\O)\times H^1_0(\O)$ 
and  the constants $C,\ga>0$ such that 
\begin{equation}\label{exp-atr}
\sup\left\{ {\rm dist}_{ H^1_0(\O)\times H^1_0(\O)}(S(t)y, \sK)\, :\; 
y\in B\right\}\le C e^{-\gamma (t-t_B)},~~~ t\ge t_B,
\end{equation}
for any bounded set $B$ from $\cH$. Moreover, we have
that $\sK\subset \sB_0$, where $\sB_0$ is the 
positively invariant set  constructed in Remark~\ref{re:abs-inv}.
\end{theorem}
\begin{proof}
We use a splitting method relying on the idea presented
in \cite{PataZelik06} (see also \cite{KaZe09}).
\par  
We first note that  it is sufficient to prove (\ref{exp-atr})
for $B=\sB_0$, where $\sB_0\subset \cH\cap ( H^1_0\times H^1_0)(\O)$
is the invariant absorbing set constructed in Remark~\ref{re:abs-inv}.
\par 
 From  (\ref{smoth-prop}) and (\ref{R-bnd}) we obviously have that
\begin{equation}\label{reg0-new}
\|\g u (t)\|^2+ \|\g u_{t}(t)\|^2+
\int_t^{t+1}\|u_{tt}(\t)\|^2d \t+\int_0^{\infty}\|\g u_{t}(\t)\|^2d \t \le C_{\sB_0},~~~ t\ge 0,
\end{equation} 
for
any solution $u(t)$ with initial data  $(u_0;u_1)$
from  $\sB_0$.
 Thus  we need only to show that 
there exists a ball $B=\{u\in  (H^2\cap H^1_0)(\O):  \|\D \|\le \rho\}$
which attracts in $H^1_0(\O)$ 
 any solution $u(t)$  satisfying (\ref{reg0-new}) with uniform
 exponential rate. 
\par 
We denote $\s(t)=\s(\|\g u(t)\|^2)$
and $\phi(t)= \phi(\|\g u(t)\|^2)$.
Since  both $\s$ and $\phi$ are strictly positive,
we  have that
\[
0<c_1\le \s(s), \phi(s)\le c_2,~~~ t\ge 0,
\]
where the constants $c_1$  and $c_2$ depend  only 
on the size of the absorbing set $\sB_0$.
Let $\nu>0$ be a parameter (which we choose large enough).
Assume that $w(t)$ solves the  problems  
\begin{equation}\label{split1}
    \left\{ \begin{array}{l}
         -\s(t)\D w_t - \phi(t)\D w+\nu w+ f(w)=h_u(t)\equiv -u_{tt}+\nu u+ h(x),~~ x\in\O,\; t>0,
\\ [2mm]
w|_{\d\O}=0,~~~ w(0)=0.
    \end{array} \right.
\end{equation}
Then  one can see that  $v(t)=w(t)-u(t)$
satisfies the equation
\begin{equation}\label{split2}
    \left\{ \begin{array}{l}
         -\s(t)\D v_t - \phi(t)\D v+\nu v+ f(w+v)-f(w)=0,~~ x\in\O,\;
 t>0,\\ [2mm]
v|_{\d\O}=0,~~~ v(0)=u_0. 
    \end{array} \right.
\end{equation}
As in the proof of Proposition~\ref{pr:dis-h2}
using the multiplier $-\D w$ in (\ref{split1}) one can see that
\[
\hf \frac{d}{dt}\left[ \s(t) \|\D w(t)\|^2 \right]
+
\phi(t) \|\D w(t)\|^2\le 
  \left[ \eps+ C_\eps \|\g u_t(t)\|^2\right]\|\D w(t)\|^2 
 +C_\eps \|h_u(t)\|^2   
\]
for all $t> 0$. 
Therefore using Gronwall's type argument
and the bounds in (\ref{reg0-new}) we obtain that
\begin{equation}\label{split3}
\|\D w(t)\|^2\le C\int_0^t e^{-\gamma (t-\t)} \|h_u(\t)\|^2 d\t
 \le C_{\sB_0}, ~~~ \forall\, t\ge 0.   
\end{equation}
where $C_{\sB_0}>0$ does not depends on $t$.
\par 
Multiplying (\ref{split2}) by $v$ in a similar way we obtain 
\[
\hf \frac{d}{dt}\left[ \s(t) \|\g v(t)\|^2 \right]
+
\phi(t) \|\g v(t)\|^2\le 
  \left[ \eps+ C_\eps \|\g u_t(t)\|^2\right]\|\g v(t)\|^2,~~~ t\ge 0, 
\]
which implies that
\begin{equation}\label{split4}
\|\g v(t)\|^2\le C \|\g u(0)\|^2 e^{-2\gamma t},~~~ t\ge 0.   
\end{equation}
Let $\cB= \{u\in H^2(\O)\cap H^1_0(\O) :  \|\D u\|^2\le C_{\sB_0}\}$,
where $C_{\sB_0}$ is the constant from (\ref{split3}).
It follows from  (\ref{split3}) and (\ref{split4}) that
\begin{equation}\label{split5}
{\rm dist}_{H_0^1(\O)}(u(t), \cB) =\inf_{b\in \cB} \|w(t)+v(t)- b\|_1
\le  \|v(t)\|_1\le C  e^{-\gamma t}, ~~~ t\ge 0.
\end{equation}
This implies the existence of the   set $\sK$
desired in the statement of the Theorem~\ref{th:ak}. 
\end{proof}
Now we consider the set $\sB_0$ defined in Remark~\ref{re:abs-inv}
 as a topological space equipped with  the partially strong 
topology (see Definition~\ref{de:ss-top}). Since 
$\sB_0$ bounded in $\cH\cap (H^1_0\times H^1_0)(\O)$, this topology
can be defined by the metric
\begin{equation}\label{metric}
\cR(y,y^*)=\|u_0-u_0^*\|_1 + \|u_1-u_1^*\|+\sum_{n=1}^\infty 
2^{-n}\frac{|(u_0-u_0^*, g_n)|}{1+|(u_0-u_0^*, g_n)|}
\end{equation}
for $y=(u_0;u_1)$  and $y^*=(u^*_0;u^*_1)$ from $\sB_0$,
where $\{ g_n\}$ is  a sequence in $L_{(p+1)/p}(\O)\cap H^{-1}(\O)$ 
such that
$\|g_n\|_{-1}=1$ and Span$\{ g_n : n\in\N \}$ is dense 
 in $L_{(p+1)/p}(\O)$.
\begin{corollary}\label{co:ak}
Let the hypotheses of Theorem~\ref{th:ak} be in force and
 $\sK$ and $\sB_0$ be the same sets as in Theorem~\ref{th:ak}. Then
there exist $C,\gamma>0$ such that 
\begin{equation}\label{exp-atr2}
\sup\left\{
\inf_{z\in\sK} \cR(S(t)y, z)\, :\; y\in \sB_0
\right\} \le C e^{-\gamma t} ~~~for ~all~~ t\ge 0.
\end{equation}
\end{corollary}
\begin{proof}
As in the proof of Theorem~\ref{th:ak} using the splitting
given by (\ref{split1}) and (\ref{split2}) we have that
\begin{align*}
\inf_{z\in\sK} \cR(S(t)y, z) &\le \|v(t)\|_1
+\sum_{n=1}^\infty 
2^{-n}\frac{|(v(t), g_n)|}{1+|(v(t), g_n)|}
\le   \|v(t)\|_1+  2^{-N+1}+\sum_{n=1}^N 
2^{-n}\frac{|(v(t), g_n)|}{1+|(v(t), g_n)|}
\\
&\le  \|v(t)\|_1\left[ 1+\sum_{n=1}^N 
2^{-n} \|g_n\|_{-1}\right]+  2^{-N+1}\le  2 \|v(t)\|_1+  2^{-N+1}
\end{align*}
for every $N\in\N$,
where $S(t)y=(u(t);u_t(t))$ with $y=(u_0;u_1)\in\sB_0$,
and $v$ solves  (\ref{split2}). 
We can choose $N= [t]$, where $[t]$ denotes integer part of $t$. 
Thus (\ref{exp-atr2}) follows from
 (\ref{split5}). 
\end{proof}   
\subsection{Global attractor in partially strong  topology}
We recall 
the  notion of  a global attractor
and some dynamical characteristics  for the  
semigroup $S(t)$ which depend  on a choice of the topology in
 the phase space 
(see, e.g., \cite{BabinVishik, Chueshov,Hale,Temam} for the general theory).
\par 
A bounded  set $\Ac\subset \cH$   is said to be 
a {\em global partially strong attractor} for $S(t)$ if 
(i) $\Ac$ is closed with respect to the partially strong 
(see Definition~\ref{de:ss-top}) topology,
(ii) $\Ac$ is strictly invariant ($S(t)\Ac=\Ac$ for all $t>0$),
and (iii) $\Ac$  uniformly  attracts in the partially strong topology
all other bounded  sets: for any (partially strong) 
vicinity $\cO$ of $\Ac$ and for any bounded set 
$B$ in $\cH$ there exists $t_*=t_*(\cO,B)$ 
such that $S(t) B\subset \cO$ for all $t\ge t_*$.
\par 
{\em Fractal dimension} $\dim^X_f M$ of a compact set $M$ in a
 complete
metric space $X$ is defined as
\[
\dim^X_fM=\limsup_{\eps\to 0}\, \frac{\ln N(M,\eps)}{\ln (1/\eps)}\;,
\]
where $N(M,\eps)$ is the minimal number of closed sets in $X$ of
diameter $2\eps$ which cover~$M$.
\par
We also  recall (see, e.g., \cite{BabinVishik})
 that the \textit{unstable set} $\mathbb{M}_+(\cN)$ emanating from 
some set $\cN\subset\cH$ is a subset
 of $\cH$ such that for each $z\in\mathbb{M}_+(\cN)$
there exists a full trajectory $\{y(t): t\in\R\}$ satisfying
$u(0) = z$ and ${\rm dist}_\cH(y(t),\cN) \to  0$ as $t\to -\8$.
\par 
Our first main result in this section is the following theorem.
\begin{theorem}\label{th1:attractor}
Let Assumptions~\ref{A1} and \ref{A:dis} be in force.
Assume also that (i) $\phi(s)$ is strictly positive
(i.e., $\phi(s)>0$  for all $s\in\R_+$) 
and (ii) $f'(s)\ge -c$ for all $s\in\R$ in 
the non-super critical case (when  (\ref{f-supercrit}) does not
hold). 
Then the semigroup $S(t)$ given by (\ref{evol-sgr}) possesses
a  global partially strong attractor $\mathfrak{A}$ in the space 
$\cH$. Moreover,  $\mathfrak{A}\subset \cH_1= 
[H^2\cap H^1_0](\O)\times H^{1}_0(\O)$ and 
\begin{equation}\label{reg-new}
\sup_{t\in\R}\left( \|\D u(t)\|^2 
+\|\g u_t(t)\|^2+ \|u_{tt}(t)\|_{-1}^2+
\int_t^{t+1}\|u_{tt}(\t)\|^2d \t \right)\le C_{\Ac}
\end{equation} 
for
any  full trajectory $\gamma=\{ (u(t); u_t(t)): t\in \R\}$ 
from the attractor
$\Ac$.
We also have that
\begin{equation}\label{str-reg}
\Ac=\mathbb{M}_+(\cN),~~~ where ~~
 \cN=\{ (u;0)\in\cH : \phi(\|\cA^{1/2}u\|^2)\cA u +f(u)=h\}.
\end{equation}
\end{theorem}
\begin{proof}
Since  $\sB_0$ is an absorbing positively invariant set (see Remark~\ref{re:abs-inv}), to prove the theorem it is sufficient
to consider the restriction of $S(t)$ on the metric space 
$\sB_0$ endowed with the metric $\cR$ given by (\ref{metric}).
By Corollary~\ref{co:ak} the dynamical system 
$(\sB_0,S(t))$ is asymptotically compact.
 Thus (see, e.g., \cite{BabinVishik,Chu98,Temam}) 
this system possesses a compact (\wrt the metric $\cR$) 
global attractor $\Ac$ which belongs to $\sK$. 
It is clear that $\Ac$ is a global partially strong attractor
for $(\cH,S(t))$ with the regularity properties stated in 
(\ref{reg-new}). 
\par
The attractor $\Ac$ is a strictly invariant compact set
in $\cH$.
By  Remark~\ref{re:gradient}  the semigroup $S(t)$ is
 gradient on  $\Ac$. Therefore 
the standard results on gradient systems with compact attractors
 (see, e.g., \cite{BabinVishik,Chueshov,Temam}) yields (\ref{str-reg}).
Thus the proof of Theorem~\ref{th1:attractor} is complete.
\end{proof}
To obtain the result on dimension for the attractor $\Ac$ we need
 the following 
amplification  of the requirements listed in the first 
part of Assumption~\ref{A1}. 
\begin{assumption}\label{A-super} 
{\rm
The functions  $\s$ and   $\phi$  belong to  $C^1(\R_+)$
and possess the properties:
\begin{itemize}
\item
[{\bf (i)}] $\s(s)>0$ and $\phi(s)>0$ for all $s\in\R_+$;
 \item
[{\bf (ii)}] 
 $\la_1\hat\mu_\phi+\mu_f>0$, where  
$\hat\mu_\phi$ is defined in \eqref{44a}, $\mu_f$
is given by (\ref{f-coercive}) and
 $\la_1$ is the first eigenvalue of  the minus Laplace operator
 in $\O$ with the Dirichlet boundary conditions 
(in the supercritical case this requirement holds automatically).  
\end{itemize} 
}
\end{assumption}

\begin{theorem}\label{th:dim}
Let Assumptions~\ref{A1}(ii), \ref{A:dis} and 
\ref{A-super}  be in force and $\inf_{s\in\R}f'(s)> -\infty$ in the non-supercritical case. 
Then the   global partially 
strong attractor $\mathfrak{A}$ 
given by Theorem~\ref{th1:attractor}
 has a finite fractal
dimension as a compact set in $\cH_r :=
[H^{1+r}\cap H^1_0](\O)\times H^{r}(\O)$ 
for every $r<1$.
\end{theorem}

Our main ingredient of the proof is the following 
weak quasi-stability estimate.

\begin{proposition}[Weak quasi-stability]\label{pr:w-q-st}
Assume that the hypotheses of Theorem~\ref{th:dim}
are in force. 
Let $u^1(t)$ and $u^2(t)$ be two weak solutions
such that $\|u^i(t)\|_2^2+ \|u^i_t(t))\|^2_1\le R^2$, 
for all $t\ge 0$,
$i=1,2$. Then
their difference $z(t)=u^1(t)-u^2(t)$ satisfies the relation
\begin{align}
\label{w-stab}
\|z_t(t)\|^2_{-1}+ \|\g z(t)\|^2  \leq &\; a_R\left(
 \|z_t(0)\|^2_{-1}+ \|\g z(0)\|^2\right) e^{-\ga_Rt}
\\ &
\; +b_R\int_0^t e^{-\ga_R(t-\t)}\left[ \| z(\t)\|^2+ 
\|\cA^{-l} z_t(\t)\|^2 
\right]d\t,\nonumber 
\end{align}
where $a_R, b_R,\ga_R$ are positive constants and $l\ge 1/2$ 
can be taken arbitrary.
\end{proposition}
\begin{proof}
Our additional hypothesis on  $\phi$ and also
the bounds for solutions $u^i$ imposed allow us to improve
the argument which led to (\ref{dif-bnd}). 
\par
Since 
\begin{equation}\label{phi-tild1}
|\widetilde\phi_{12}(t)| |(\g (u^1+u^2),\g z)|^2
\le C_R\|z\|^2,~~~ t\ge 0,
\end{equation}
for our case, it follows from Lemmas~\ref{le:z-mult} and 
\ref{le:f-coerc} that 
\begin{align}\label{key-for-qst}
\frac{d}{dt}\left[ (z,z_t)+\frac14 \s_{12}(t)
\cdot \|\g z\|^2\right]+ \frac12 \phi_{12}(t)\cdot \|\g z\|^2  +\,
c_0\left[
\int_\O(|u^1|^{p-1}+|u^2|^{p-1}) |z|^2 dx \right]  \\
\le  \; \|z_t\|^2+  C_R\left(\|\g u^1_t\|+|\g u^2_t\|\right) \|\g z\|^2 +
C\|z\|^2, \nonumber
\end{align}
where $c_0=0$ in the non-supercritical case.
Now
as in the proof of Theorem~\ref{pr:wp1} we use the multiplier 
$\cA^{-1}z_t$. However now our  considerations 
of the term $|(G(u^1,u^2; t),\cA^{-1} z_t)|$
of the form (\ref{g-zt-min})
involves the additional positivity type requirement imposed on $\phi$. 
\par
Using the inequality 
$\|\cA^{-1/2} z_t\|^2\le \eta \| z_t\|^2 + C_\eta\|\cA^{-l} z_t\|^2$
for any $\eta>0$ and $l\ge 1/2$,
one can  see that 
\begin{eqnarray*}
|G_1(t)|\le \eps \|z_t\|^2+  C_{R,\eps}
\left(\|\g u^1_t\|^2+\|\g u^2_t\|^2\right) \|\g z\|^2
\end{eqnarray*}
and also, involving (\ref{phi-tild1}),
\begin{eqnarray*}
|G_2(t)|\le \eps \|z_t\|^2+
C_{R,\eps} \left[\|\cA^{-l} z_t\|^2 +\|z\|^2\right] 
\end{eqnarray*}
for any $\eps>0$ and for every $l\ge 1/2$.
Therefore from (\ref{f-a-1}) we obtain that  
\begin{eqnarray*}
|(G(u^1,u^2; t),\cA^{-1} z_t)| &
\le & C_{R,\eps}\left[\left(\|\g u^1_t\|^2+\|\g u^2_t\|^2\right) \|\g z\|^2
+ \|z\|^2+  \|\cA^{-l} z_t\|^2\right]  \\
 & & +\,
 \eps  \left[ \|z_t\|^2+
c_0\int_\O(|u^1|^{p-1}+|u^2|^{p-1}) |z|^2 dx \right] 
\end{eqnarray*}
for any $\eps>0$,
where $c_0=0$ in the non-supercritical case.
 Consequently by (\ref{ar:mult1}) and (\ref{key-for-qst}) the function $\Psi(t)$ given by
(\ref{psi-t-w}) satisfies the relation
\begin{eqnarray*}
\lefteqn{
 \frac{d\Psi}{dt}
+ \frac{\eta}2 \phi_{12}(t)\cdot \|\g z\|^2 
  +\left[\hf \s_{12}(t)-\eta-\eps\right]\| z_t\|^2
} \\ &&
+\, c_0(\eta -\eps) 
\int_\O(|u^1|^{p-1}+|u^2|^{p-1}) |z|^2 dx 
  \le C_\eps(R)\left[
d_{12}(t)
 \|\g z\|^2
  +
\|\cA^{-l} z_t\|^2+\|z\|^2\right], 
\end{eqnarray*}
where 
$d_{12}(t)=\|\g u^1_t(t)\|^2+\|\g u^2_t(t)\|^2$.
Therefore after an appropriate choice of $\eta$ and $\eps$ we have that
\[
\frac{d\Psi}{dt}+ \alpha_{12}(t) \Psi 
\le c_R \left[
\|\cA^{-l} z_t\|^2+\|z\|^2\right]~~~\mbox{with} ~~
\a_{12}(t)=\frac{\eta}2\phi_{12}(t)- c_R d_{12}(t),
\]
This implies  that
\begin{equation}\label{psi-est}
\Psi(t)\le c_R\exp\left\{ \! -\!\! \int_0^t\!\! \alpha_{12}(\t)d\t \right\} \Psi(0)+
c_R\int_0^t\exp\left\{\! - \!\! \int_\t^t\!\! \alpha_{12}(\xi)d\xi \right\}\left[
\|\cA^{-l} z_t(\t)\|^2+\|z(\t)\|^2\right]d\t. 
\end{equation}
Under Assumption~\ref{A-super} by Remark~\ref{re:e+}
we have    estimate  (\ref{R-bnd}) which yields
 that
\[
\int_\t^t \alpha_{12}(\xi)d\xi\ge \eta \phi_R\cdot (t-\t)-c_R
\int_\t^t d_{12}(\xi)d\xi\ge \eta \phi_R\cdot (t-\t)- C_{R}~~
\]
for all $t>\t\ge 0$. with positive $\phi_R$ and $C_R$. Thus
from (\ref{psi-est}) and (\ref{eq-nrms}) we obtain (\ref{w-stab}).
\end{proof}
\begin{lemma}\label{le:w-q-st}
Let the hypotheses of Proposition~\ref{pr:w-q-st} be in force. Then
the difference $z(t)=u^1(t)-u^2(t)$ of two  weak solutions satisfies the relation
\begin{equation}
\label{ztt-eeq}
\int_0^T\|\cA^{-l}z_{tt}(\t)\|^2  d\t   \leq  C_R\left(
 \|z_t(0)\|^2_{-1}+ \|\g z(0)\|^2\right) 
 +C_R T \int_0^T\left[ \| z(\t)\|^2+ 
\|\cA^{-l} z_t(\t)\|^2 
\right]d\t
\end{equation}
for every $T\ge1$,
where $C_R>0$ is a constant and $l\ge 3/2$ is  arbitrary such that 
$L_1(\O)\subset H^{-2l}(\O)$, i.e. $l>d/4$.
\end{lemma}
\begin{proof}
It follows from (\ref{abs-dif}) that
$\|\cA^{-l}z_{tt}\|\le C_R(  \|\cA^{-l+1}z\|+ \|\cA^{-l+1}z_{t}\|)
+\|\cA^{-l}G(u^1,u^2;t)\|$.
By the embedding $L_1(\O)\subset  H^{-2l}(\O)$ we obviously have that 
\begin{eqnarray*}
\|\cA^{-l}G(u^1,u^2;t)\| &\le & C_R  \|\cA^{1/2}z\|+
C \int_\O |f(u^1)-f(u^2)| dx  \\
& \le &  C_R  \|\cA^{1/2}z\|+
C \int_\O \left(1+ |u^1|^{p-1}+ |u^2|^{p-1}\right) |z| dx.  
\end{eqnarray*}
Therefore using (\ref{dif-bnd}) (and also (\ref{dif-bnd+})
 in the supercritical case) we obtain that 
\[
\int_a^b\|\cA^{-l}z_{tt}(\t)\|^2  d\t   \leq  C_R\left(
 \|z_t(a)\|^2_{-1}+ \|\g z(a)\|^2\right) 
\]
for every $a<b$ such that $b-a\le 1$.
Therefore 
\begin{eqnarray*}
\int_0^T\|\cA^{-l}z_{tt}(\t)\|^2  d\t  &  \leq &
 \sum_{k=0}^{[T]-1}\int_k^{k+1}\|\cA^{-l}z_{tt}(\t)\|^2  d\t +\int_{[T]}^T\|\cA^{-l}z_{tt}(\t)\|^2  d\t  \\ 
 & \le &  
 C_R
\sum_{k=0}^{[T]}\left(
 \|z_t(k)\|^2_{-1}+ \|\g z(k)\|^2\right), 
\end{eqnarray*}
where $[T]$ denotes the integer part of $T$.  Now we can apply 
the stabilizability estimate in (\ref{w-stab})
with $t=k$ for each $k$ and obtain  (\ref{ztt-eeq}).
\end{proof}
\subsubsection*{Proof of Theorem~\ref{th1:attractor}} 
We use the idea due to  
M\'alek--Ne\v{c}as~\cite{malek-ne} (see also  \cite{malek} and \cite{ChuLas}).
\par For some $T\ge 1$ which we specify latter and for some 
$l>\max\{d,6\}/4$ 
we consider  the space
\[
W_T=\left\{u\in  C(0,T;  H^{1}_0(\O))\, :\; u_t\in  C(0,T;  H^{-1}(\O)), \;
 u_{tt}\in  L_2(0,T;  H^{-2l}(\O))   \right\}
\]
with the norm
\[
|u|_{W_T}^2=\max_{t\in [0,T]}\left[ \|\g u(t)\|^2+ \|u_t(t)\|_{-1}^2\right] +
\int_0^T \|u_{tt}(t)\|_{-2l}^2 dt.
\]
Let
${\Ac}_T$ be  the set of weak solutions
to (\ref{main_eq}) on the interval $[0,T]$
with initial data  $(u(0);u_t(0))$ from the attractor  $\Ac$. It is clear that   
${\Ac}_T$ is a closed bounded set in $W_T$.
Indeed, if the sequence of solutions $u^n(t)$ with initial data in
${\Ac}_T$
is fundamental in $W_T$, then we have that $u^n(0)\to u_0$ strongly 
in $H^1(\O)$,  $u^n(0)\to u_0$  weakly in $L_{p+1}(\O)$
 and $u^n_t(0)\to u_1$ weakly in $L_{2}(\O)$ for some $(u_0;u_1)\in \Ac$.
By (\ref{dif-bnd}) and (\ref{ztt-eeq}) this implies that $u^n(t)$ converges 
in $W_T$ to the solution with initial data   $(u_0;u_1)$.
This yields  the closeness of ${\Ac}_T$ in $W_T$.
The boundedness of ${\Ac}_T$ is obvious.
\par 
On ${\Ac}_T$ we define the shift operator $V$
by the formula
\[
V\, :\; {\Ac}_T\mapsto {\Ac}_T, ~~~ [Vu](t)=u(T+t),~~  t\in [0,T].
\]
It is  clear that ${\Ac}_T$ is strictly invariant \wrt $V$, i.e.
$V{\Ac}_T={\Ac}_T$.
It follows from (\ref{dif-bnd}) and (\ref{ztt-eeq}) that 
\begin{equation*}
|V U_1-VU_2|_{W_T}\le C_T | U_1- U_2|_{W_T},~~~ U_1,U_2\in\widetilde{\sB}_T.
\end{equation*}
By Proposition~\ref{pr:w-q-st} we  have that
\begin{eqnarray*}
\max_{s\in [0,T]}\left\{
\|z_t(T+s)\|^2_{-1}+ \|\g z(T+s)\|^2\right\} & \leq & a
 e^{-\ga T}\max_{s\in [0,T]}
\left\{
 \|z_t(s)\|^2_{-1}+ \|\g z(s)\|^2\right\}
\\ & &
\; +\; b \int_0^{2T} \left[ \| z(\t)\|^2+ 
\|\cA^{-l} z_t(\t)\|^2 
\right]d\t,
\end{eqnarray*}
where $a,b,\ga>0$ depends on the size of the set $\Ac$ in $\cH_1= 
[H^2\cap H^1_0](\O)\times H^{1}_0(\O)$.
Lemma~\ref{le:w-q-st} and Proposition~\ref{pr:w-q-st}  also
yield that
\[
\int_T^{2T}\|\cA^{-l}z_{tt}\|^2  d\t   \leq  C  e^{-\ga T}\left(
\|z_t(0)\|^2_{-1}+ \|\g z(0)\|^2\right) 
 +C (1+ T) \int_0^{2T}\left[ \| z\|^2+ 
\|\cA^{-l} z_t\|^2 
\right]d\t.
\]
Therefore we obtain that 
\begin{equation}\label{qst-w-est}
|V U_1-VU_2|^2_{W_T}\le q_{T} | U_1- U_2|^2_{W_T} +C_T
\left[ n_T^2(U_1-U_2) +n_T^2(VU_1-VU_2)\right]
\end{equation}
for every $ U_1,U_2\in {\Ac}_T$, where
$q_T= C e^{-\gamma T}$ and 
 the seminorm $n_T(U)$ has the form
\[
n^2_T(U) \equiv \int_0^{T}\left[ \| u\|^2+ 
\|\cA^{-l} u_t\|^2 
\right]d\t~~\mbox{for}~~ U=\{u(t)\}\in  W_T. 
\]
One can see that this seminorm is compact on $W_T$.
Therefore  we can choose $T\ge 1$ such that $q_T<1$ in (\ref{qst-w-est})
 and apply Theorem 2.15\cite{ChuLas}
to conclude that   $\Ac_T$   has a finite fractal dimension in $W_T$.
One can also see that
$\Ac=\{ (u(t);u_t(t))_{t=s}\, :\, u(\cdot)\in \Ac_T\} $
does not depend on $s$. Therefore the fractal    
dimension of  ${\Ac}$ is finite in the space 
$\tilde\cH= H^{1}_0(\O)\times  H^{-1}(\O)$. 
By interpolation argument it follows from (\ref{reg-new}) and 
(\ref{dif-bnd}) that $S(t)\big|_\Ac$
is a H\"{o}lder continuous mapping from $\tilde \cH$ into $\cH_r$ 
for each $t>0$.
Since  ${\rm dim}_f^{\tilde \cH}\Ac<\infty$,
this implies that ${\rm dim}_f^{\cH_r}\Ac$ is finite.
\par

\subsection{Attractor in the energy space. Non-supercritical case}
In this section we deal with the attractor in the 
strong topology of the energy space which we understand 
in the standard sense (see, e.g., \cite{BabinVishik, Chueshov,Hale,Temam}).
Namely,
the \textit{global attractor}  of the evolution semigroup  $S(t)$
is defined as a bounded closed  set $\Ac\subset \cH$
which is strictly invariant ($S(t)\Ac=\Ac$ for all $t>0$) and  uniformly  attracts
all other bounded  sets:
$$
\lim_{t\to\8} \sup\{{\rm dist}_\cH(S(t)y,\Ac):\ y\in B\} = 0
\quad\mbox{for any bounded  set $B$ in $\cH$.}
$$
\par
Since $\cH=H^1_0(\O)\times L_2(\O)$ in the 
non-supercritical case, Theorem \ref{th:ak} implies the existence
of a compact set in $\cH$ which attracts bounded sets
 in the strong topology.
This leads to the following  assertion.

\begin{theorem}\label{th:attractor}
Let Assumptions~\ref{A1} and \ref{A:dis} be in force.
Assume also that $\phi(s)$ is strictly positive
(i.e., $\phi(s)>0$  for all $s\in\R_+$) 
and $f'(s)\ge -c$ for all $s\in\R$ in 
the non-supercritical case (when the bounds in (\ref{f-supercrit})
are not valid). 
Then the evolution semigroup $S(t)$ possesses 
a compact global attractor $\Ac$ in $\cH$.
This attractor $\Ac$ coincides with the partially strong attractor 
given by Theorem~\ref{th1:attractor} and thus 
(i)~ $\mathfrak{A}\subset \cH_1= 
[H^2\cap H^1_0](\O)\times H^{1}_0(\O)$; (ii)~the 
relation in \eqref{reg-new}
hold;
(iii)~$\Ac=\mathbb{M}_+(\cN)$, where  $\cN$  is the set of equilibria 
(see (\ref{str-reg})). Moreover,
 we have that 
\begin{equation}\label{grad-cnv}
{\rm dist}_\cH(y,\cN)\to 0 ~~ as~ t\to\infty~ for ~any ~y\in\cH.
\end{equation}
If in addition we assume  Assumption~\ref{A-super}(ii), then 
~$\Ac$
 has a finite  fractal dimension in
the space   
 $\cH_r =
[H^{1+r}\cap H^1_0](\O)\times H^{r}(\O)$
 for every $r<1$.
\end{theorem}
\begin{proof}
We apply Theorems~\ref{th1:attractor} and \ref{th:dim}.
To obtain (\ref{grad-cnv}) we only note that by  Remark~\ref{re:gradient}  the semigroup $S(t)$ is
 gradient  on the {\it whole} space $\cH$. Thus 
the standard results on gradient systems 
 (see, e.g., \cite{BabinVishik,Chueshov,Temam}) 
lead to the conclusion in (\ref{grad-cnv}).
\end{proof}
Under additional hypotheses we can establish other dynamical properties 
of the system under the consideration. We impose now the following
set of requirements.
\begin{assumption}\label{A3-crit}
{\rm We assume that $\phi\in C^2(\R_+)$ is a  nondecreasing 
function ($\phi'(s)\ge 0$ for $s\ge 0$),  $f'(s)\ge -c$
for some $c\ge 0$, 
and one of the following requirements
fulfills:
 \begin{itemize}
    \item[{\bf (a)}]
  either $f$ is subcritical:
  either $d\le2$ or (\ref{f-crit}) holds with $p< p_*\equiv (d+2)(d-2)^{-1}$, 
 $d\ge 3$;
 \item[{\bf (b)}]
or else  $3\le d\le 6$, 
 $f\in C^2(\R)$ is critical, i.e.,  
\begin{equation*}
|f''(u)|\le C\left(1+|u|^{p_*-2}\right),~~~u\in\R,~~~ p_*= (d+2)(d-2)^{-1}.  
\end{equation*}
\end{itemize} 
}
\end{assumption}
Our second main result  in this section is the following theorem.
\begin{theorem}\label{th:exp-det}
Let Assumptions~\ref{A1}(ii), \ref{A-super}, and \ref{A3-crit} 
 be in force.
Then 
\begin{enumerate}
\item[{\bf (1)}] Any trajectory $\gamma=\{ (u(t); u_t(t)): t\in \R\}$ from the attractor $\Ac$ given by Theorem~\ref{th:attractor}
possesses the properties 
\begin{equation}\label{u-smth}
(u; u_t;u_{tt})\in L_\infty (\R; [H^2\cap H^1_0](\O)\times H^1_0(\O)\times L_2(\O))
\end{equation}
and there is $R>0$ such that 
\begin{equation}\label{u-smth2}
\sup_{\ga \subset \Ac}\sup_{t\in\R}\left( \|\D u(t)\|^2
+\|\g u_t(t)\|^2+ \|u_{tt}(t)\|^2\right)\le R^2.
\end{equation} 
\item[{\bf (2)}] There exists a 
fractal exponential attractor $\Ac_{exp}$  in $\cH$.
\item[{\bf (3)}]
 Let   ${\cal L} =  \{ l_j : j= 1,...,N\}$ be a finite set of
functionals on $H^1_0(\O)$ and  
\[
\epsilon_{\cal L}=\epsilon_{\cal L}(H^1_0(\O), L_2(\O))\equiv
\sup\left\{ \|u\|\,:\, u\in H^1_0(\O), ~~l_j(u)=0,
j=1,...,N,
\|u\|_1\le1\right\}
\]
be the corresponding completeness
defect.; 
Then  there exists $\eps_0>0$ such that
under the condition $\epsilon_\cL\le \eps_0$
 the set $\cL$ is (asymptotically)
determining in the sense   that the property
\[
\lim_{t\to\infty}\max_{j} \int_t^{t+1}|l_j(u^1(s)-u^2(s))|^2ds=0
\]
implies that $\lim_{t\to\infty} \|S(t)y_1-S(t)y_2\|_\cH=0$.
Here above $S(t)y_i=(u^i(t); \.u^i(t))$, $i=1,2$.
\end{enumerate}
\end{theorem}
We recall (see, e.g., \cite{EFNT94} and
 also \cite{Chueshov,ChuLas,cl-book})
that a
 compact set $\Ac_{\rm exp}\subset \cH$
is said to be a  fractal exponential
attractor
for  the dynamical system $(\cH, S(t))$  iff  
$\Ac_{\rm exp}$ is a positively invariant set
of finite fractal dimension in $\cH$ and
for every bounded set $D\subset \cH$ there exist positive constants
$t_D$, $C_D$ and $\gamma_D$ such that
\begin{equation*}
d_X\{S(t)D\, |\, \Ac_{\rm exp}\}\equiv
\sup_{x\in D} \mbox{dist}\,_\cH  (S(t)x,\, \Ac_{\rm exp})\le C_D\cdot e^{-\gamma_D(t-t_D)},
\quad t\ge t_D.
\end{equation*}
We also mentioned that the notion of determining functionals goes back
to the papers  by Foias and Prodi~\cite{FP68}
and by Ladyzhenskaya~\cite{Lad75} for the 2D Navier-Stokes equations.
For the further  development of the theory  we refer to
 \cite{cjt97} and to the survey \cite{Chu98}, see also
 the references quoted in these publications.
We note that for the first time determining functionals for 
second order (in time) evolution equations with a nonlinear damping was considered in 
\cite{CHuKal2001}, see also a discussion in \cite[Section 8.9]{cl-book}. 
We also refer  to \cite{Chu98} and \cite[Chap.5]{Chueshov} for a description of sets
of functionals with small completeness defect.
\subsubsection*{Proof of Theorem~\ref{th:attractor}}
The main ingredient of the proof 
is some quasi-stability  property of $S(t)$
in the energy space $\cH$
which is stated in the following assertion.

\begin{proposition}[Strong quasi-stability]\label{pr:s-q-st}
Suppose that
 Assumptions~\ref{A1}(ii), \ref{A-super} and \ref{A3-crit} hold. 
Let $u^1(t)$ and $u^2(t)$ be two weak solutions
such that $\|(u^i(0); u^i_t(0))\|_\cH\le R$, $i=1,2$, then
their difference $z(t)=u^1(t)-u^2(t)$ satisfies the relation
\begin{equation}
\label{s-stab-str}
\|z_t(t)\|^2+ \|\g z(t)\|^2  \leq  a_R\left(
 \|z_t(0)\|^2+ \|\g z(0)\|^2\right) e^{-\ga_Rt}
 +b_R\int_0^t e^{-\ga_R(t-\t)} \| z(\t)\|^2 
d\t, 
\end{equation}
where $a_R, b_R,\ga_R$ are positive constants.
\end{proposition}
\begin{proof}
As a starting point we consider the energy type relation (\ref{ener-dif})
for the difference $z$ (which we already
use in the proof of the second part of Proposition~\ref{pr:gener}) 
and estimate
the term
\[
 G(t)\equiv(G(u^1,u^2; t), z_t)=H_1(t)+ H_2(t)+H_3(t)
\]
given by  (\ref{G-zt-new}) using the additional hypotheses imposed.
One can see that 
\[
|H_1(t)|\le  \eps \|\g z_t\|^2+  C_{R,\eps} (\|\g u^1_t\|^2+\| \g u^2_t\|^2) \|\g z\|^2.
\]
Here and below we  use the fact that 
$\|u^i_t(t)\|^2+ \|\g u^i(t)\|^2  \leq  C_R$ for all $t\ge 0$ 
(see (\ref{R-bnd})).
\par
We also have that
\begin{eqnarray*}
H_2(t)&= & \hf \frac{d}{dt}\left[ \widetilde\phi_{12}(t) 
 |(\g (u^1+u^2),\g z)|^2 \right] +\hat H_2(t),
\end{eqnarray*}
where 
$|\hat H_2(t))|\le   C_{R} (\|\g u^1_t\|+\| \g u^2_t\|) \|\g z\|^2$.
\par
If  $f$ is subcritical, i.e., Assumption~\ref{A3-crit}(a) holds,
 then the estimate for
$H_3(t)$ is direct:
\[
|H_3(t)|\le C_R\|\g z_t\|\|z\|_{1-\delta}\le \eps\left( \|\g z_t\|^2+ \|\g z\|^2\right) +
  C_{R,\eps} \| z\|^2
\]
for some $\delta>0$ and for any $\eps>0$. Therefore in 
the argument below we concentrate on the critical case described
in Assumption~\ref{A3-crit}(b).  In this case
 we have that 
\begin{eqnarray*}
H_3(t)&= & \hf \frac{d}{dt}\left[\int_0^1\int_\O f'(u^2+\la (u^1-u^2))
|z|^2 d\la dx \right] +\hat H_3(t),
\end{eqnarray*}
where
\[
\hat H_3(t)=- \hf \int_0^1\int_\O f''(u^2+\la (u^1-u^2))
(u^2_t+\la (u^1_t-u^2_t))|z|^2 d\la dx. 
\]
By the growth condition of $f''$ we have that 
\[
|\hat H_3(t)|\le C \int_\O \left[ 1+ |u^1|^{p_*-2} +|u^2|^{p_*-2}\right]
(|u^1_t|+|u^2_t|)|z|^2  dx. 
\]
Therefore the  H\"older inequality and the Sobolev embedding  
$H^1(\O)\subset L_{p_*+1}(\O)$ imply that 
\begin{eqnarray*}
|\hat H_3(t)| &\le& C \left[ 1+ \|u^1\|_{L_{p_*+1}(\O)}^{p_*-2} +\|u^2\|_{L_{p_*+1}(\O)}^{p_*-2}\right]
\left[\|u^1_t\|_{L_{p_*+1}(\O)}+\|u^2_t\|_{L_{p_*+1}(\O)}\right]
\|z\|_{L_{p_*+1}(\O)}^2\\
 &\le& C_R \left[\|\g u^1_t\|+\|\g u^2_t\|\right]
\|\g z\|^2.
\end{eqnarray*}
Now we introduce the energy type functional
\begin{eqnarray*}
E_*(t) &= & \frac12 \|z_t\|^2+
\frac14 \phi_{12}(t) \|\g z\|^2 \\
& &+\; \hf \left[\int_0^1\int_\O f'(u^2+\la (u^1-u^2))
|z|^2 d\la dx + \widetilde\phi_{12}(t) 
 |(\g (u^1+u^2),\g z)|^2 \right].
\end{eqnarray*}
From (\ref{ener-dif}) and the calculations above we obviously have that 
\begin{eqnarray*}
\frac{d}{dt}E_*(t) +\left[\hf\s_{12}(t)-\eps\right]\|\g z_t\|^2
\le C_{R,\eps}\left[ d_{12}(t)+\sqrt{d_{12}(t)}\right] \|\g z\|^2,
\end{eqnarray*}
where
$d_{12}(t)=\|\g u^1_t(t)\|^2+\|\g u^2_t(t)\|^2$.
Therefore using Lemma~\ref{le:z-mult}
we obtain that the function
\[
W_*(t)= E_*(t)+\eta \left[ (z,z_t)+\frac14
\s_{12}(t) \|\g z\|^2\right],~~\eta>0,
\]
satisfies the relation
\begin{eqnarray*}
\frac{d}{dt}W_*(t) +\left[\hf\s_{12}(t)-\eps\right]\|\g z_t\|^2   
- \eta  \|z_t\|^2+ \eta \left[
\frac12 \phi_{12}(t) \|\g z\|^2  + \widetilde\phi_{12}(t) 
 |(\g (u^1+u^2),\g z)|^2  \right]\\
+\; \eta  \int_0^1\int_\O f'(u^2+\la (u^1-u^2))
|z|^2 d\la dx  
\le \eps \|\g z\|^2+C_{R,\eps} d_{12}(t) \|\g z\|^2.
\end{eqnarray*}
Therefore, if we introduce $\tilde W(t)=W_*(t)+C\|z(t)\|^2$
 with appropriate $C>0$ and  with $\eta>0$ small enough, then we obtain 
that
\[
a_R \left( \|z_t(t)\|^2 +\|\g z(t)\|^2 \right) \le \tilde W(t)
\le b_R  \left( \|z_t(t)\|^2 +\|\g z(t)\|^2 \right)
\]
 and
\[
\frac{d}{dt}\tilde W(t)+c_R\tilde W(t)\le C_R d_{12}(t)\|\g z\|^2+C\|z(t)\|^2
\]
with positive constants.
Thus  the finiteness of  the integral in (\ref{R-bnd}) and 
the standard Gronwall's argument implies the result in (\ref{s-stab-str})
in the critical case. In the subcritical case we use the same argument
but for the functional $E_*$ without the term containing $f'$. 
\end{proof}
{\bf Completion of the proof of Theorem~\ref{th:attractor}:} 
 Proposition~\ref{pr:s-q-st} means that the semigroup
 $S(t)$ is quasi-stable on the absorbing set $\sB_0$ defined in
Remark~\ref{re:abs-inv}  
in the sense of Definition 7.9.2~\cite{cl-book}.
Therefore
to obtain the result on regularity stated in 
(\ref{u-smth}) and
(\ref{u-smth2}) we first apply  Theorem~7.9.8~\cite{cl-book}
which gives us  that 
\[
\sup_{t\in\R}\left( 
\|\g u_t(t)\|^2+ \|u_{tt}(t)\|^2\right)\le C_\Ac 
~~\mbox{ 
for any
trajectory $\gamma=\{ (u(t); u_t(t)): t\in \R\}\subset\Ac$}. 
\]
Applying (\ref{reg-new}) we obtain (\ref{u-smth}) and
(\ref{u-smth2}).
\par 
By (\ref{smoth-prop}) any weak solution $u(t)$ 
possesses the property
\[
\int_t^{t+1}\|u_{tt}(\tau)\|^2d\t \le C_{R,T}
~~\mbox{for}~~ t\in [0,T],~~\forall\, T >0,
\]
provided $(u_0;u_1)\in S(1)\sB_0$,
where $\sB_0$ is
 the absorbing set  defined in
Remark~\ref{re:abs-inv}. 
This implies that $t\mapsto S(t)y$ is a $1/2$-H\"older continuous function
with values in $\cH$ for every $y\in S(1)\sB_0$.
Therefore the existence of a
 fractal exponential attractor follows from
 Theorem~7.9.9~\cite{cl-book}. 
\par   
To prove the statement concerning determining functionals 
we use the same idea as in the proof of 
 Theorem~8.9.3~\cite{cl-book}.

\end{document}